\newif\ifsiam     
\newif\ifnummat   
\newif\ifmoc      
\newif\ifblank
\newif\ifima

\siamfalse
\nummatfalse
\mocfalse
\blankfalse
\imafalse

\imatrue

\ifsiam
    \documentclass[review]{siamart190516}
    \newtheorem{remark}{Remark}[section]
\fi

\ifnummat
    \RequirePackage{fix-cm}
    \documentclass{svjour3}
    \journalname{Numerische Mathematik}
    \smartqed
    \DeclareMathAlphabet{\mathcal}{OMS}{cmsy}{m}{n}
\fi

\ifmoc
    \documentclass{amsart}
    \newtheorem{theorem}{Theorem}[section]
    \newtheorem{lemma}[theorem]{Lemma}
    \newtheorem{corollary}[theorem]{Corollary}

    \theoremstyle{definition}

    \theoremstyle{remark}
    \newtheorem{remark}[theorem]{Remark}

    \numberwithin{equation}{section}

\fi

\ifima
    \documentclass{imanum}
    \usepackage{color}
\fi

\numberwithin{theorem}{section}
\numberwithin{equation}{section}
    \usepackage[all]{xy}
    \usepackage{amssymb,amsmath,stmaryrd}
    \usepackage{amscd}
    \usepackage{braket,amsfonts,amsopn}
    \usepackage{bm}
    \usepackage{makecell,multirow,diagbox}
    \usepackage{mathrsfs}
    \usepackage{graphicx,epstopdf}
    \usepackage{subfigure}
\ifmoc
    \usepackage[pdftex,colorlinks]{hyperref}
\fi

\DeclareMathOperator{\divg}{div}

\DeclareMathOperator{\curl}{curl}
\DeclareMathOperator{\rot}{rot}

\newcommand{\ab}[2]{\langle#1,#2\rangle}

\newcommand{\Th}{\mathcal{T}_{h}}

\newcommand{\CT}{\mathcal{T}}

\newcommand{\vertiii}[1]{{\left\vert\kern-0.25ex\left\vert\kern-0.25ex\left\vert #1
    \right\vert\kern-0.25ex\right\vert\kern-0.25ex\right\vert}}

\def\ltzAuthor{Ludmil T. Zikatanov}

\def\yulAuthor{Yuwen Li}

\ifsiam
    
    \def\yulShortAuthor{Y.~Li}
\else
    \def\ltzShortAuthor{L. T. Zikatanov}
    \def\yulShortAuthor{Y. Li}
\fi

\def\yulAddress{Department of Mathematics, Pennsylvania State University, University Park, PA 16802.}

\def\ltzEmail{ludmil@psu.edu}
\def\yulEmail{yuwli@psu.edu}

\title{Residual-based a posteriori error estimates of mixed methods for {a three-field}  Biot's consolidation model}

\def\shortTitle{A posteriori estimates of mixed methods for Biot's model}

\def\myKeywords{TBD}

\def\myAMS{65N12, 65N15, 65N30}

\def\myAbstract{We present residual-based a posteriori error estimates of mixed finite element methods for the three-field formulation of Biot's consolidation model. The error estimator is an upper and lower bound of the space time discretization error up to data oscillation. As a by-product, we also obtain new a posteriori error estimate of mixed finite element methods for the heat equation.}

\begin{document}


\ifsiam
  \author{\yulAuthor%
         \thanks{\yulAddress {Email:}\, \yulEmail.}
}
  \maketitle

  \begin{abstract}\myAbstract\end{abstract}
  \pagestyle{myheadings}
  \thispagestyle{plain}
  \markboth{\yulShortAuthor}{\shortTitle}

\fi


\ifnummat
   \author{
          \yulAuthor%
           }
  \institute{
              \yulShortAuthor : \yulAddress\, {Email:}\yulEmail}
  \date{Received: 23 November 2019\  / Accepted: date}
  \maketitle
  \begin{abstract}\myAbstract\end{abstract}
  \begin{subclass}\myAMS \end{subclass}
  \markboth{
  \yulShortAuthor }{\shortTitle}
\fi


\ifmoc
    \bibliographystyle{amsplain}
    \author[\yulShortAuthor,\ \ltzShortAuthor]{\yulAuthor,\ \ltzAuthor}
    \address{\yulAddress}
    \email{\yulEmail,\ \ltzEmail}

    \subjclass[2010]{Primary \myAMS}
    \date{23 November 2019}
    \begin{abstract}\myAbstract\end{abstract}
    \maketitle
    \markboth{
    \yulShortAuthor,\ \ltzShortAuthor }{\shortTitle}
\fi



\ifima
    \bibliographystyle{IMANUM-BIB}
    \author{%
    {\sc
    Yuwen Li \thanks{Corresponding author. Email: yuwli@psu.edu},
    Ludmil T.~Zikatanov\thanks{Email: ludmil@psu.edu}
    } \\[2pt]
    Department of Mathematics, Pennsylvania State University, University Park, PA 16802, USA}
    \shortauthorlist{Y.~Li,\ and\ L.~T.~Zikatanov}

    \maketitle

    \begin{abstract}
    {\myAbstract}
    {Biot's consolidation model, a posteriori error estimates, reliability, efficiency, heat equation}
    \end{abstract}
\fi

\section{Introduction}
The mathematical modeling of poro-elastic materials is aimed at describing the
interactions between the deformation and fluid flow in a
fluid-saturated porous medium. In this paper we provide a posteriori error estimators for the fully discrete, time dependent Biot's consolidation model for poroelastic media. A pioneering model of poroelasticity in one-dimensional setting was given in~\cite{terzaghi}.
Nowadays, the popular formulations are in three-dimensions and they follow the model by Maurice Biot in several works, e.g., \cite{biot1,biot2}. The system of partial differential equations describing the Biot's consolidation model has a great deal of applications in geomechanics,
petroleum engineering, and biomechanics.

The two-field formulation of Biot's consolidation model is classical and has
been investigated in e.g., \cite{Zen1984,Showalter2000,Murad1996,ErnSeb2009}.
Three-field formulations, which include an unknown Darcy velocity,
several conforming and non-conforming discretizations involving
Stokes-stable finite-element spaces have been recently proposed  as a viable approach for discretization of the Biot's model.
Various three field formulations were considered in \cite{Wheeler2007I,Wheeler2007II} with
and a priori error estimates are presented in such a work.
Recenly,  three-field formulation using Stokes stable elements,  based on displacement, pressure, and
total pressure was proposed and analysed in~\cite{OyRu2016}. A nonconforming discretization, which also provides element-wise mass conservation, is found  in~\cite{HRGZ2017}. Parameter robust analysis using three field discontinuous Galerkin formulation is
given in~\cite{HongKraus2018}, where a general theory for the a priori error analysis was introduced.
Other stable discretizations and solvers are presented in e.g.,~\cite{LJJ2016,LMW2017,RHOAGZ2018}. Readers are referred to \cite{LJJ2016} for parameter robust error analysis for four- and five-field formulations. Finite volume and finite difference discretizations have also been used in this field and we point to~\cite{Gaspar2003,Gaspar2006,nordbotten_FVM} for more results and references on such methods for Biot's system. We note that our further considerations are restricted to the finite element method and we will not discuss finite difference and finite volumen methods here.

There are a few works on a posteriori error control for the fully discretized time-dependent problem, see, e.g., \cite{Johnson1991,Johnson1995,Picasso1998,Verfurth2003,Makridakis2003,Lakkis2006,Ern2015,Ern2019} for a posteriori error estimates of the primal formulation of the heat equation. A posteriori error estimation of the mixed formulation of the heat equation can be found in e.g., \cite{CFA2006,Ern2010,Larson2011,Memon2012,KPS2018}. For the classical two-field formulation in Biot's consolidation model, residual, equilibrated, and functional error estimators are derived in \cite{ErnSeb2009,RPEGK2017,Kumar2018}. In addition, equilibrated error estimators are {developed in \cite{ARN2019,ARN2020} for the four- and five-field formulations and the fixed stress splitting scheme.} Comparing to the equilibrated error indicators, residual error estimators are simpler to implement and do not require solving auxiliary problems on local patches. Several space-time adaptive algorithms based on residual error estimators are proven to be convergent, see, e.g., \cite{ChenFeng2004,Kreuzer2012,Kreuzer2019}.

A main result in our paper is the construction of the reliable residual-based a posteriori error estimator for the three field Biot's system. To the best of our knowledge, there are no such error estimators for the mixed formulations of the Biot's model using more than two fields. Formulations using more than two fields have conservation properties which makes them practically interesting, however,  their analysis is more challenging. In this paper, we derive residual a posteriori error estimates for the three-field formulation and prove that the estimator is reliable, that is, it provides an upper bound of the space-time error in the natural variational norm. Since the three-field formulation directly approximates the flux $\bm{w}\in H(\divg,\Omega)$, special attention must be paid to energy estimates and the residual in the dual space $H(\divg,\Omega)^\prime$, which is a major obstacle in the construction of such error estimators. The analysis presented here with the help of regular decomposition and commuting quasi-interpolations, however, successfully tackles such problems, see Theorems \ref{regqi} and \ref{semiresult} for details.

Another main result of this paper is the lower bound in Theorem \ref{lower}. As far as we know, existing residual, equilibrated, and functional error estimators in Biot's consolidation model are not shown to be lower bounds of the space-time discretization error. This is partly due to the complexity of the Biot's model equations. Motivated by Verf\"urth 's technique introduced in \cite{Verfurth2003}, we split the residual and estimator into space and time parts. The temporal estimator can be controlled by the spatial estimator and discretization error, while the spatial estimator is in turn controlled by the finite element error and a small portion of the temporal estimator, where the ``smallness" is due to a weight function in time. The details are given later in Section~\ref{seclower}.

Since the three-field formulation of Biot's consolidation model \eqref{var} contains the mixed formulation of the heat equation \eqref{heat}, we review existing a posteriori error estimates of mixed methods for the heat equation. Using a duality argument, \cite{CFA2006} first obtained $L^2(0,T;H(\divg,\Omega)^\prime)$ a posteriori estimates of the flux variable and $L^\infty(0,T;L^2(\Omega))$ estimates of the potential in mixed methods for the heat equation. Using the idea of elliptic reconstruction proposed by \cite{Makridakis2003}, the works \cite{Larson2011,Memon2012} presented $L^2(0,T;L^2(\Omega))$- and $L^\infty(0,T;L^2(\Omega))$-type a posteriori error estimates of the flux variable in mixed methods for the heat equation. However,
there is no proof that the estimators proposed in \cite{CFA2006,Larson2011,Memon2012} provide lower bounds of the discretization error.
On the other hand,  \cite{Ern2010} presented an equilibrated estimator with a lower bound for the error in post-processed potential based on the $L^2(0,T;H^1(\Omega))\cap H^1(0,T;H^{-1}(\Omega))$-norm. Their estimator does not control the error in the flux variable. Comparing to the aforementioned error estimators, a posteriori analysis in this paper indeed yields a new estimator for the mixed discretization of the heat equation that is both an upper and lower bound of the space time error in the natural norm, see Section \ref{seclower} for details.

The rest of this paper is organized as follows. In Section \ref{secenergy}, we present preliminaries and derive energy estimate for the three field formulation of Biot's consolidation model. Section \ref{secsemi} is devoted to a posteriori error estimates of a semi-discrete scheme \eqref{semidis}. In Section \ref{secfully}, we develop a posteriori error estimator of the fully discrete scheme \eqref{fully} and prove its reliability. In Section \ref{seclower}, we show that the error estimators are lower bounds of the space-time error and present a posteriori estimates of mixed methods for the heat equation. In Section \ref{secexp}, we present numerical experiments validating our theoretical results. Section \ref{secconc} is for concluding remarks.

\section{Preliminaries and Energy estimates}\label{secenergy}
Given a $\mathbb{R}^d$-valued function $\bm{u}$, the symmetric gradient $\bm{\varepsilon}$ and stress tensor $\bm{\sigma}$ are
$$\bm{\varepsilon}(\bm{u}):=\frac{1}{2}(\nabla\bm{u}+(\nabla\bm{u})^T),\quad\bm{\sigma}(\bm{u}):=2\mu\bm{\varepsilon}(\bm{u})+\lambda(\divg\bm{u})\bm{I},$$
where $\mu>0, \lambda>0$ are Lam\'e coefficients, $\bm{I}$ is the $d\times d$ identity matrix. Let $\Omega$ be a Lipschiz domain in $\mathbb{R}^d$ and $T>0$ be the final time. The three-field formulation of the Biot's consolidation model reads
\begin{subequations}\label{Biot}
    \begin{align}
    -\divg\bm{\sigma}(\bm{u})+\alpha\nabla p&=\bm{f}\text{ in }\Omega\times(0,T],\label{Biot1}\\
    \partial_t(\beta p+\alpha\divg\bm{u})+\divg\bm{w}&=g\text{ in }\Omega\times(0,T],\label{Biot2}\\
   \bm{K}^{-1}\bm{w}+\nabla p&=\bm{0}\text{ in }\Omega\times(0,T],\label{Biot3}
\end{align}
\end{subequations}
subject to the initial condition $\bm{u}(0)=\bm{u}_0$, $p(0)=p_0$ in $\Omega.$
For the simplicity of presentation, we consider homogeneous boundary conditions
\begin{align*}
    \bm{u}&=\bm{0}\text{ on }\Gamma_1\times(0,T],\quad\bm{\sigma}(\bm{u})\bm{n}=0\text{ on }\Gamma_2\times(0,T],\\
    p&=0\text{ on }\Gamma_2\times(0,T],\quad(\bm{K}\nabla p)\cdot\bm{n}=0\text{ on }\Gamma_1\times(0,T],
\end{align*}
where $\partial\Omega=\Gamma_1\cup\Gamma_2,$  $\Gamma_1\cap\Gamma_2=\emptyset,$ and $\bm{n}$ denotes the outward unit normal to $\partial\Omega.$ Note that the Neumann boundary condition for $p$ on $\Gamma_1$ imposes an essential boundary condition for $\bm{w}$ on $\Gamma_1$. In addition, we assume $\alpha$, $\beta$ are constants and $\bm{K}=\bm{K}(\bm{x})$ is a  time-independent and uniformly elliptic matrix-valued function, i.e.,
\begin{equation*}
    C_1|\bm{\xi}|^2\leq\bm{\xi}^T\bm{K}(\bm{x})\bm{\xi}\leq C_2|\bm{\xi}|^2\text{ for all }\bm{\xi}\in\mathbb{R}^d\text{ and }\bm{x}\in\Omega,
\end{equation*}
where $C_1, C_2$ are positive constants. We introduce function spaces where we seek a weak solution to the system given in~\eqref{Biot}:
\begin{align*}
    \bm{V}&=\{\bm{v}\in [H^1(\Omega)]^d: \bm{v}=0\text{ on }\Gamma_1\},\quad
    Q=L^2(\Omega),\\ \bm{W}&=\{\bm{w}\in [L^2(\Omega)]^d: \divg\bm{w}\in L^2(\Omega),~ \bm{w}\cdot\bm{n}=0\text{ on }\Gamma_1\}.
\end{align*}
Let $(\cdot,\cdot)$ denote the $L^2(\Omega)$ inner product for scalar-, vector-, or matrix-valued functions. Next, we introduce several bilinear forms:
\begin{align*}
    a(\bm{u},\bm{v})&:=(\bm{\sigma}(\bm{u}),\bm{\varepsilon}(\bm{v})),\quad b(\bm{v},q):=(\alpha\divg\bm{v},p),\\
    c(p,q)&:=(\beta p,q),\quad d(\bm{z},q):=(\divg\bm{z},q),\quad e(\bm{w},\bm{z}):=(\bm{K}^{-1}\bm{w},\bm{z}).
\end{align*}
The norms associated with the bilinear forms given above are
\begin{align*}
\|\bm{v}\|^2_a&:=a(\bm{v},\bm{v}),\quad\|q\|_c^2:=c(q,q),\\
\|\bm{z}\|_e^2&:=e(\bm{z},\bm{z}),\quad
    \|\bm{z}\|_{\bm{W}}^2:=\|\bm{z}\|_e^2+\|\divg\bm{z}\|,
\end{align*}
where $\|\cdot\|$ denotes the $L^2(\Omega)$ norm.
For the spaces defined earlier we have the following correspondence with the norms: $\bm{V}$ is equipped with the  $\|\cdot\|_a$-norm, $Q$ is equipped with $\|\cdot\|_c$-norm, and $\bm{W}$ is equipped with the $\bm{W}$-norm. Because we are dealing with a time-dependent problem, we need the spaces of Hilbert-valued
functions as follows:
Given a Hilbert space $H$, we define
\begin{align*}
L^\infty(0,T;H)&=\{v: v(t)\in H\text{ for }t\in T,~{\text{ess}\sup}_{0\leq t\leq T}\|v(t)\|_{H}<\infty\},\\
L^2(0,T;H)&=\{v: v(t)\in H\text{ for }t\in T,~\int_0^T\|v(t)\|_{H}^2dt<\infty\},\\
    H^1(0,T;H)&=\{v\in L^2(0,T;H): \partial_tv\in L^2(0,T;H)\},
\end{align*}
see, e.g., \cite{Evans2010} for more details.
The variational formulation of \eqref{Biot} then is to find $\bm{u}\in H^1(0,T;\bm{V})$, $p\in H^1(0,T;Q)$, and $\bm{w}\in L^2(0,T;\bm{W})$ such that $\bm{u}(0)=u_0$, $p(0)=p_0$ and
\begin{subequations}\label{var}
\begin{align}
    a(\bm{u},\bm{v})-b(\bm{v}, p)&=(\bm{f},\bm{v}),\\
    c(\partial_tp,q)+b(\partial_t\bm{u},q)+d(\bm{w},q)&=(g,q),\\
    e(\bm{w},\bm{z})-d(\bm{z},p)&=0
\end{align}
\end{subequations}
for all $\bm{v}\in\bm{V},  q\in Q$, and $\bm{z}\in\bm{W}$ a.e.~$t\in(0,T]$. It can be observed that \eqref{var} with $\bm{u}=\bm{v}=\bm{0}$ reduces to the mixed formulation of the heat equation or time-dependent Darcy flow:
\begin{subequations}\label{heat}
\begin{align}
    c(\partial_tp,q)+d(\bm{w},q)&=(g,q),\quad q\in Q,\\
    e(\bm{w},\bm{z})-d(\bm{z},p)&=0,\quad\bm{z}\in\bm{W}.
\end{align}
\end{subequations}

In the rest of this section, we establish an
energy estimate of \eqref{var} which is the main tool for deriving a posteriori error estimates. The well-posedness of two-field formulation can be found in e.g., \cite{Zen1984,Showalter2000}. For the three-field formulation we have the following result.
\begin{theorem}\label{existence}
Let $u_0\in\bm{V}$, $p_0\in Q$, $\bm{f}\in H^1(0,T;\bm{V}^\prime)$, and $g\in L^2(0,T;Q)$. Then the variational formulation \eqref{var} admits a unique weak solution
$$(\bm{u},p,\bm{w})\in H^1(0,T;\bm{V})\times H^1(0,T;Q)\times L^2(0,T;\bm{W}).$$
\end{theorem}
We skip the proof of Theorem \ref{existence} as it directly follows from the energy estimates in Lemma \ref{energy} and a standard argument using a Galerkin method in space, in the same fashion as for the linear parabolic equation (see, e.g., \cite{Evans2010}).
For the purpose of a posteriori error estimation, we consider a more general variational problem: Find $\tilde{\bm{u}}\in H^1(0,T;\bm{V})$, $\tilde{p}\in H^1(0,T;Q)$, $\tilde{\bm{w}}\in L^2(0,T;\bm{W})$, such that
\begin{subequations}\label{Gvar}
    \begin{align}
    a(\tilde{\bm{u}},\bm{v})-b(\bm{v}, \tilde{p})&=\ab{\bm{F}_1}{\bm{v}},\quad\bm{v}\in\bm{V},\label{Gvar1}\\
    c(\partial_t\tilde{p},q)+b(\partial_t\tilde{\bm{u}},q)+d(\tilde{\bm{w}},q)&=\ab{F_2}{q},\quad q\in Q,\label{Gvar2}\\
    e(\tilde{\bm{w}},\bm{z})-d(\bm{z},\tilde{p})&=\ab{\bm{F}_3}{\bm{z}},\quad\bm{z}\in\bm{W},\label{Gvar3}
\end{align}
\end{subequations}
where $\bm{F}_1\in H^1(0,T;\bm{V}^\prime)$,  ${F}_2\in L^2(0,T;Q^\prime)$, $\bm{F}_3\in H^1(0,T;\bm{W}^\prime)$ are time-dependent bounded linear functionals living in dual spaces.
At each time $t\in[0,T]$, the dual norms are given by
\begin{align*}
    \|\bm{F}_1\|_{\prime}=\|\bm{F}_1\|_{\bm{V}^\prime}&:=\sup_{\bm{v}\in\bm{V},\|\bm{v}\|_a=1}\ab{\bm{F}_1}{\bm{v}},\\
    \|{F}_2\|_{\prime}=\|{F}_2\|_{Q^\prime}&:=\sup_{q\in Q,\|q\|_c=1}\ab{\bm{F}_2}{q},\\
    \|\bm{F}_3\|_{\prime}=\|\bm{F}_3\|_{\bm{W}^\prime}&:=\sup_{\bm{z}\in\bm{W},\|\bm{z}\|_{\bm{W}}=1}\ab{\bm{F}_3}{\bm{z}}.
\end{align*}
Norms of $\partial_t\bm{F}_1\in\bm{V}^\prime$ and $\partial_t\bm{F}_3\in\bm{W}^\prime$ are defined in a similar fashion.
Given $t\in[0,T]$ and an interval $I\subseteq[0,T],$  we make use of the norms
\begin{align*}
    \vertiii{(\tilde{\bm{u}},\tilde{p},\tilde{\bm{w}})(t)}^2&:=\|\tilde{\bm{u}}(t)\|_a^2+\|\tilde{p}(t)\|_c^2+\|\tilde{\bm{w}}(t)\|_e^2,\\
    \|(\tilde{\bm{u}},\tilde{p},\tilde{\bm{w}})\|^2_{L^2(I;X)}&:=\int_I\big(\|\tilde{\bm{u}}\|_a^2+\|\partial_t\tilde{\bm{u}}\|_a^2+\|\tilde{p}\|_c^2+\|\partial_t\tilde{p}\|_c^2+\|\tilde{\bm{w}}\|_{\bm{W}}^2+{\|\partial_t\tilde{\bm{w}}\|_{\bm{W}^\prime}^2}\big)ds.
\end{align*}
The following energy estimate is crucial to a posteriori error estimation of numerical methods for \eqref{var}.
\begin{lemma}\label{energy}
There exists a constant $C_{\text{stab}}$ dependent only on $\mu$, $\alpha,$ $\beta$, $\bm{K}$, $\Omega$ such that for all $t\in(0,T]$,
\begin{equation*}
    \begin{aligned}
    &\vertiii{(\tilde{\bm{u}},\tilde{p},\tilde{\bm{w}})(t)}^2+\|(\tilde{\bm{u}},\tilde{p},\tilde{\bm{w}})\|^2_{L^2(0,t;X)}\leq C_{\text{stab}}\big\{\vertiii{(\tilde{\bm{u}},\tilde{p},\tilde{\bm{w}})(0)}^2\\
    &+\left(\int_0^t\|F_2\|_{\prime}ds\right)^2+\int_0^t\big(\|\bm{F}_1\|_{\prime}^2+\|\partial_t\bm{F}_1\|_{\prime}^2+\|F_2\|_{\prime}^2+\|\bm{F}_3\|_{\prime}^2+\|\partial_t\bm{F}_3\|_{\prime}^2\big)ds\big\}.
    \end{aligned}
\end{equation*}
\end{lemma}
\begin{proof}
Setting $\bm{v}=\partial_t\tilde{\bm{u}}$, $\bm{z}=\tilde{\bm{w}}$, $q=\tilde{p}$ in \eqref{Gvar} yields
\begin{equation}\label{id1}
    \frac{1}{2}\frac{d}{dt}\|\tilde{\bm{u}}\|_a^2+\frac{1}{2}\frac{d}{dt}\|\tilde{p}\|_c^2+\|\tilde{\bm{w}}\|_e^2=\ab{\bm{F}_1}{\partial_t\tilde{\bm{u}}}+\ab{F_2}{\tilde{p}}+\ab{\bm{F}_3}{\tilde{\bm{w}}}.
\end{equation}
On the other hand, differentiating \eqref{Gvar1} and \eqref{Gvar3} with respect to time $t$ gives
\begin{align*}
    a(\partial_t\tilde{\bm{u}},\bm{v})-b(\bm{v}, \partial_t\tilde{p})&=\ab{\partial_t\bm{F}_1}{\bm{v}},\\
    e(\partial_t\tilde{\bm{w}},\bm{z})-d(\bm{z},\partial_t\tilde{p})&=\ab{\partial_t\bm{F}_3}{\bm{z}}.
\end{align*}
Taking as test functions $\bm{v}=\partial_t\tilde{\bm{u}}$ and $\bm{z}=\tilde{\bm{w}}$ in the equations above and using \eqref{Gvar2} with $q=\partial_t\tilde{p}$ then leads to
\begin{equation}\label{id2}
    \|\partial_t\tilde{\bm{u}}\|_a^2+\|\partial_t\tilde{p}\|_c^2+\frac{1}{2}\frac{d}{dt}\|\tilde{\bm{w}}\|_e^2=\ab{\partial_t\bm{F}_1}{\partial_t\tilde{\bm{u}}}+\ab{F_2}{\partial_t\tilde{p}}+\ab{\partial_t\bm{F}_3}{\tilde{\bm{w}}}.
\end{equation}
Using \eqref{id1}, \eqref{id2}, the Cauchy--Schwarz and Young's inequalities, we obtain
\begin{align*}
    &\frac{1}{2}\frac{d}{dt}\|\tilde{\bm{u}}\|_a^2+\frac{1}{2}\frac{d}{dt}\|\tilde{p}\|_c^2+(1-\delta)\|\tilde{\bm{w}}\|_e^2+\frac{1}{2}\|\partial_t\tilde{\bm{u}}\|_a^2\\
    &\quad+\frac{1}{2}\|\partial_t\tilde{p}\|_c^2+\frac{1}{2}\frac{d}{dt}\|\tilde{\bm{w}}\|_e^2\leq G+\|F_2\|_\prime\|\tilde{p}\|_c+\delta\|\divg\tilde{\bm{w}}\|^2,
\end{align*}
where $\delta>0$ and
$$G=\|\bm{F}_1\|_{\prime}^2+\|\partial_t\bm{F}_1\|_{\prime}^2+\frac{1}{2}\|F_2\|_{\prime}^2+\frac{\delta^{-1}}{2}\|\bm{F}_3\|_{\prime}^2+\frac{\delta^{-1}}{2}\|\partial_t\bm{F}_3\|_{\prime}^2.$$
Integrating the previous inequality yields
\begin{equation}\label{est1}
    \begin{aligned}
    &\frac{1}{2}\vertiii{(\tilde{\bm{u}},\tilde{p},\tilde{\bm{w}})(t)}^2+\int_0^t\big(\frac{1}{2}\|\partial_t\tilde{\bm{u}}\|_a^2+\frac{1}{2}\|\partial_t\tilde{p}\|_c^2+(1-\delta)\|\tilde{\bm{w}}\|_e^2\big)ds\\
    &\quad\leq\frac{1}{2}\vertiii{(\tilde{\bm{u}},\tilde{p},\tilde{\bm{w}})(0)}^2+\int_0^t\big(G+\|F_2\|_{\prime}\|\tilde{p}\|_c+\delta\|\divg\tilde{\bm{w}}\|^2\big)ds.
\end{aligned}
\end{equation}
Recall that
$\|\tilde{p}\|_{L^\infty(0,t;Q)}:=\max_{0\leq s\leq t}\|\tilde{p}(s)\|_c.$
In particular, \eqref{est1} implies that
\begin{equation*}
    \begin{aligned}
    &\frac{1}{2}\|\tilde{p}(s)\|_c^2\leq\frac{1}{2}\vertiii{(\tilde{\bm{u}},\tilde{p},\tilde{\bm{w}})(0)}^2\\
    &\quad+\int_0^t\big(G+\delta\|\divg\tilde{\bm{w}}\|^2\big)ds+\|\tilde{p}\|_{L^\infty(0,t;Q)}\int_0^t\|F_2\|_{\prime}ds
\end{aligned}
\end{equation*}
for all $0\leq s\leq t$. Hence a combination of the previous estimate with
$$\|\tilde{p}\|_{L^\infty(0,t;Q)}\int_0^t\|F_2\|_{\prime}ds\leq\frac{1}{4}\|\tilde{p}\|^2_{L^\infty(0,t;Q)}+\left(\int_0^t\|F_2\|_{\prime}ds\right)^2$$ shows that
\begin{equation}\label{est2}
    \begin{aligned}
    &\frac{1}{4}\|\tilde{p}\|_{L^\infty(0,t;Q)}^2\leq\frac{1}{2}\vertiii{(\tilde{\bm{u}},\tilde{p},\tilde{\bm{w}})(0)}^2\\
    &\quad+\int_0^t\big(G+\delta\|\divg\tilde{\bm{w}}\|^2\big)ds+\left(\int_0^t\|F_2\|_\prime ds\right)^2.
\end{aligned}
\end{equation}
Using \eqref{est1} and \eqref{est2} and a Young's inequality, we obtain
\begin{equation}\label{est3}
    \begin{aligned}
    &\frac{1}{2}\vertiii{(\tilde{\bm{u}},\tilde{p},\tilde{\bm{w}})(t)}^2+\int_0^t\big(\frac{1}{2}\|\partial_t\tilde{\bm{u}}\|_a^2+\frac{1}{2}\|\partial_t\tilde{p}\|_c^2+(1-\delta)\|\tilde{\bm{w}}\|_e^2\big)ds\\
    &\quad\leq\vertiii{(\tilde{\bm{u}},\tilde{p},\tilde{\bm{w}})(0)}^2+2\int_0^t\big(G+\delta\|\divg\tilde{\bm{w}}\|^2\big)ds+2\left(\int_0^t\|F_2\|_\prime ds\right)^2,
\end{aligned}
\end{equation}
Let $C$ be a generic constant dependent only on $\alpha,$ $\beta$, $\mu$, $\Omega$. It follows from \eqref{Gvar2} with $q=\divg\tilde{\bm{w}}$ that
\begin{equation}\label{est4}
 \|\divg\tilde{\bm{w}}\|^2\leq C\big(\|F_2\|^2_\prime+\|\partial_t\tilde{p}\|^2_c+\|\partial_t\tilde{\bm{u}}\|^2_a\big).
\end{equation}
Taking the derivative with respect to time on both sides of \eqref{Gvar3} shows that
{\begin{equation}\label{extra1}
 \|\partial_t\tilde{\bm{w}}\|_{\bm{W}^\prime}\leq C\big(\|\partial_t\bm{F}_3\|_\prime+\|\partial_t\tilde{p}\|_c\big).
\end{equation}}
The inf-sup condition for $d(\cdot,\cdot)$ together with \eqref{Gvar3} and \eqref{Gvar1} then imply the following inequality:
\begin{equation}\label{Gvar13}
\|\tilde{p}\|_c+\|\tilde{\bm{u}}\|_a\leq C\big(\|\bm{F}_1\|_\prime+\|\bm{F}_3\|_\prime+\|\tilde{\bm{w}}\|_e\big).
\end{equation}
Choosing a sufficiently small $\delta>0$ and combining \eqref{est3}
and \eqref{est4}--\eqref{Gvar13} completes the proof of the lemma.
\end{proof}

\section{Error estimator for the semi-discrete problem}\label{secsemi}
Let $\Th$ be a conforming simplicial triangulation of $\Omega$ that is aligned with $\Gamma_1$ and $\Gamma_2$. The mesh $\Th$ is shape-regular in the sense that
$$\max_{K\in\Th}\frac{r_K}{\rho_K}:=\widetilde{C}_{\text{shape}}<\infty,$$
where $r_K$, $\rho_K$ are radii of circumscribed and inscribed spheres of $K$. Let $\bm{V}_h\subset\bm{V}$, $\bm{W}_h\subset\bm{W}$, $Q_h\subset Q$ be suitable finite element spaces based on $\Th$. In particular, we choose
$\bm{V}_h\times Q_h$ to be
a stable mixed element pair for the Stokes equation, and $\bm{W}_h\times Q_h$
to be a stable mixed element pair for the mixed formulation of Poisson's equation. It has been shown in e.g., \cite{HongKraus2018,RHOAGZ2018} that this choice leads to stable space discretization. For example, $\bm{V}_h\times Q_h$ can be chosen to be the $(P_1+\text{face bubble functions})\times P_0$ element  (see \cite{GR1986}) and $\bm{W}_h\times Q_h$ can be the lowest order Raviart--Thomas (see \cite{RT1977}) or Brezzi--Douglas--Marini element (see \cite{BDM1985}). Let $\mathcal{F}(\Th)$ denote the collection of faces in $\Th$ and $\bm{n}_F$ be a unit normal to $F$ for any face $F\in\mathcal{F}(\Th)$. Let $\mathcal{P}_k(K)$ denote the space of polynomials of degree no greater than $k$ on $K,$ and
\begin{align*}
    \bm{V}_{h,l}&=\{\bm{v}\in\bm{V}: \bm{v}|_K\in[\mathcal{P}_1(K)]^d\text{ for all }K\in\Th\},\\
    \bm{B}_h&=\{\bm{v}\in\bm{V}: \bm{v}|_K\in\text{span}\{\phi_F\bm{n}_F\}_{F\subset\partial K, F\in\mathcal{F}(\mathcal{T}_h)}\text{ for all }K\in\Th\}.
\end{align*}
Here, $\phi_F$ is the face bubble function supported on union of elements having $F\in\mathcal{F}(\Th)$ as a face, i.e.,
\(\phi_F=\prod_{z_j\in F}\lambda_j\) where $\lambda_j$ is the barycentric coordinate corresponding to the vertex $z_j$ in the face $F$.
The triple $\bm{V}_h\times Q_h\times\bm{W}_h$ can be chosen as $\bm{V}^0_h\times Q^0_h\times\bm{W}^0_h$, where
\begin{align*}
    {\bm{V}}^0_h&=\bm{V}_{h,l}\oplus\bm{B}_h,\\
    Q^0_h&=\{q\in L^2(\Omega): q|_K\in\mathcal{P}_0(K)\text{ for all }K\in\Th\},\\
    \bm{W}^0_h&=\{\bm{z}\in\bm{W}: \bm{z}|_K\in[\mathcal{P}_0(K)]^d+\mathcal{P}_0(K)\bm{x}\text{ for all }K\in\Th\}.
\end{align*}
Here $\bm{x}=(x_1,x_2,\ldots,x_d)^T$ is the linear position vector. In general, we assume the inclusion $\bm{W}_h^0\subseteq\bm{W}_h.$

The semi-discrete version of \eqref{var} is to find $\bm{u}_h\in H^1(0,T;\bm{V}_h)$, $p_h\in H^1(0,T;Q_h)$, and $\bm{w}_h\in L^2(0,T;\bm{W}_h)$ such that $\bm{u}_h(0)=u^0_{h}$, $p_h(0)=p^0_{h}$ and
\begin{subequations}\label{semidis}
    \begin{align}
    a(\bm{u}_h,\bm{v})-b(\bm{v}, p_h)&=(\bm{f},\bm{v}),\quad\bm{v}\in\bm{V}_h,\label{semidis1}\\
    c(\partial_tp_{h},q)+b(\partial_t\bm{u}_{h},q)+d(\bm{w}_h,q)&=(g,q),\quad q\in Q_h,\\
    e(\bm{w}_h,\bm{z})-d(\bm{z},p_h)&=0,\quad\bm{z}\in\bm{W}_h.\label{semidis3}
\end{align}
\end{subequations}
Here $u^0_{h}\in\bm{V}_h, p^0_{h}\in Q_h$ are some finite element approximation to $u_0$ and $p_0$.
In this section, we derive a posteriori error estimation for the semi-discrete method \eqref{semidis}. To this end, let \begin{equation}\label{error1}
    e_u=\bm{u}-\bm{u}_h,\quad e_p=p-p_h,\quad e_w=\bm{w}-\bm{w}_h.
\end{equation} It follows from \eqref{var} that the errors satisfy
\begin{subequations}\label{semierr}
    \begin{align}
    a(e_u,\bm{v})-b(\bm{v}, e_p)&=\ab{\bm{r}_1}{\bm{v}},\quad\bm{v}\in\bm{V},\label{semierr1}\\
    c(\partial_te_{p},q)+b(\partial_te_u,q)+d(e_w,q)&=\ab{r_2}{q},\quad q\in Q,\\
    e(e_w,\bm{z})-d(\bm{z},e_p)&=\ab{\bm{r}_3}{\bm{z}},\quad\bm{z}\in\bm{W},
\end{align}
\end{subequations}
where the residuals $\bm{r}_1(t)\in\bm{V}^\prime$, $r_2(t)\in Q^\prime$, $\bm{r}_3(t)\in\bm{W}^\prime$ are defined by
\begin{align*}
    \ab{\bm{r}_1}{\bm{v}}&:=(\bm{f},\bm{v})-a(\bm{u}_h,\bm{v})+b(\bm{v}, p_h),\\
    \ab{{r}_2}{q}&:=(g,q)-c(\partial_tp_{h},q)-b(\partial_t\bm{u}_{h},q)-d(\bm{w}_h,q),\\
    \ab{\bm{r}_3}{\bm{z}}&:=-e(\bm{w}_h,\bm{z})+d(\bm{z},p_h).
\end{align*}
With the help of Lemma \ref{energy}, we immediately obtain the following corollary.
\begin{corollary}\label{cor}
For the errors defined in~\eqref{error1} and $t\in(0,T]$, we have
\begin{equation*}
    \begin{aligned}
    &\vertiii{(e_u,e_p,e_w)(t)}^2+\|(e_u,e_p,e_w)\|^2_{L^2(0,t;X)}\leq C_{\text{stab}}\big\{\vertiii{(e_u,e_p,e_w)(0)}^2\\
    &\quad+\left(\int_0^t\|r_2\|_\prime ds\right)^2+\int_0^t\big(\|\bm{r}_1\|_\prime^2+\|\partial_t\bm{r}_1\|_\prime^2+\|r_2\|_\prime^2+\|\bm{r}_3\|_\prime^2+\|\partial_t\bm{r}_3\|_{\prime}^2\big)ds\big\}.
    \end{aligned}
\end{equation*}
\end{corollary}
For a $\mathbb{R}^d$-valued function $\bm{v}=(v_i)_{1\leq i\leq d}$ and a scalar-valued function $v$, let
\begin{align*}
    \curl\bm{v}&=(\partial_{x_2}{v}_3-\partial_{x_3}{v}_2,\partial_{x_3}{v}_1-\partial_{x_1}{v}_3,\partial_{x_1}{v}_2-\partial_{x_2}{v}_1)^T\text{ when }d=3,\\
    \curl v&=(\partial_{x_2}{v},-\partial_{x_1}{v})^T,\quad\rot\bm{v}=\partial_{x_1}{v}_2-\partial_{x_2}{v}_1\text{ when }d=2.
\end{align*}
Let $\bm{N}^0_h$ denote the lowest order N\'ed\'elec edge element space (see~\cite{Nedelec1980})
\begin{align*}
    \bm{N}^0_h&=\{\bm{v}\in [L^2(\Omega)]^3: \curl\bm{v}\in[L^2(\Omega)]^3,~\bm{v}\times\bm{n}=0\text{ on }\Gamma_1,\\
    &\bm{v}|_K\in[\mathcal{P}_0(K)]^3+[\mathcal{P}_0(K)]^3\times\bm{x}\text{ for all }K\in\Th\},
\end{align*}
and $V_h$ denote the \emph{scalar} linear element space
\begin{align*}
    V_h&=\{v\in H^1(\Omega): v|_K\in\mathcal{P}_1(K)\text{ for all } K\in\mathcal{T}_h, ~v=0\text{ on }\Gamma_1\}.
\end{align*}
For each $K\in\Th$, let $h_K=|K|^{\frac{1}{d}}$ denote the size of $K$, $\|\cdot\|_K$ denote the $L^2$-norm on $K$, and $\|\cdot\|_{\partial K}$ denote the $L^2$-norm on $\partial K$.
To estimate the norms of $\bm{r}_3, \partial_t\bm{r}_3\in\bm{W}^\prime,$ we need the following theorem, which is a combination of the $H^1$-regular decomposition (see e.g.,~\cite{Hip2002,PZ2002,DH2014}) and bounded quasi-interpolation operators which commute with the exterior differentiation (see~\cite{Schoberl2008,DH2014}).
\begin{theorem}\label{regqi}
Let $\Omega\subset\mathbb{R}^d$ with $d=3$ (resp.~$d=2$). There exist quasi-interpolations $\Pi_h: [L^2(\Omega)]^3\rightarrow\bm{W}^0_h$ (resp.~$\Pi_h: [L^2(\Omega)]^2\rightarrow\bm{W}^0_h$), and $J_h: [L^2(\Omega)]^3\rightarrow\bm{N}^0_{h}$ (resp.~$J_h: L^2(\Omega)\rightarrow V_{h}$) such that $\Pi_h\curl=\curl J_h$. In addition, for any $\bm{z}\in\bm{W}$, there exist $\bm{\varphi}\in\bm{V}$  (resp.~$\bm{\varphi}\in H^1(\Omega)$, $\bm{\varphi}|_{\Gamma_1}=0$) and $\bm{\phi}\in\bm{V}$, such that
$$\bm{z}=\curl\bm{\varphi}+\bm{\phi},\quad\Pi_h\bm{z}=\curl J_h\bm{\varphi}+\Pi_h\bm{\phi},$$
and
\begin{equation}\label{interperror}
    \begin{aligned}
    &\sum_{K\in\Th}\big(h_K^{-2}\|\bm{\varphi}-J_h\bm{\varphi}\|^2_K+h_K^{-2}\|\bm{\phi}-\Pi_h\bm{\phi}\|^2_K\\
    &\quad+h_K^{-1}\|\bm{\varphi}-J_h\bm{\varphi}\|^2_{\partial K}+h_K^{-1}\|\bm{\phi}-\Pi_h\bm{\phi}\|^2_{\partial K}\big)\leq C_{\text{reg}}\|\bm{z}\|_{\bm{W}}^2,
    \end{aligned}
\end{equation}
where $C_{\text{reg}}$ depends only on $\bm{K},$ $\Omega$, $\Gamma_1,$     $\widetilde{C}_{\text{shape}}$.
\end{theorem}
Theorem \ref{regqi} or its variants are widely used in the a posteriori error estimation of stationary problems based on $H(\divg)$ or $H(\curl)$, see, e.g., \cite{CNS2007,Schoberl2008,HuangXu2012,DH2014,CW2017,YL2019,YL2019b,HLMS2019}.

Now we are in a position to derive a posteriori error estimator of the system given in~\eqref{semidis}. For each face $F$, let  $\bm{n}_F$ be a unit normal to $F$ where $\bm{n}_F$ is chosen to be outward pointing when $F$ is a boundary face. For each interior $F\in\mathcal{F}(\Th)$ shared by $K_1, K_2\in\Th$ and a piecewise $H^1$-function $\chi$, let $[\chi]|_F:=(\chi|_{K_1}-\chi|_{K_2})|_F$ denote the jump across $F$, where $\bm{n}_F$ is pointing from $K_1$ to $K_2.$ For any $F\subset\partial\Omega$ that is a boundary face in $\mathcal{T}_h$, we set $[\chi]|_F:=0$ if $F\subset\Gamma_1$, and $[\chi]|_F:=\chi|_F$ if $F\subset\Gamma_2$. Regarding the mesh $\Th,$ we use the following error indicators
\begin{align*}
    \mathcal{E}^1_{\Th}(\bm{u}_h,p_h,\bm{f})&:=\sum_{K\in\Th}\big\{h_K^2\|\bm{f}+\divg\bm{\sigma}(\bm{u}_h)-\alpha\nabla p_h\|_K^2\\
    &+\sum_{F\in\mathcal{F}(\Th), F\subset\partial K}h_K\|[\bm{\sigma}(\bm{u}_h)-\alpha p_h\bm{I}]\bm{n}_F\|_F^2\big\},\\
    \mathcal{E}^2_{\Th}(\partial_t\bm{u}_{h},\partial_tp_{h},\bm{w}_h,g)&:=\sum_{K\in\Th}\|g-\beta \partial_tp_{h}-\divg\partial_t\bm{u}_{h}-\divg\bm{w}_h\|_K^2.
\end{align*}
Another error estimator is
\begin{align*}
    &\mathcal{E}^3_{\Th}(p_h,\bm{w}_h):=\sum_{K\in\Th}\big\{h_K^2\|\bm{K}^{-1}\bm{w}_h+\nabla p_h\|_K^2+h_K^2\|\curl(\bm{K}^{-1}\bm{w}_h)\|_K^2\\
    &\quad+\sum_{F\in\mathcal{F}(\Th), F\subset\partial K}h_K\|[(\bm{K}^{-1}\bm{w}_h)\times\bm{n}_F]\|_F^2+h_K\|[p_h]\|_F^2\big\}\text{ when }d=3,
\end{align*}
and when $d=2$ (for a two dimensional problem)
\begin{align*}
    \mathcal{E}^3_{\Th}(p_h,\bm{w}_h)&:=\sum_{K\in\Th}\big\{h_K^2\|\bm{K}^{-1}\bm{w}_h+\nabla p_h\|_K^2+h_K^2\|\rot(\bm{K}^{-1}\bm{w}_h)\|_K^2\\
    &\quad+\sum_{F\in\mathcal{F}(\Th), F\subset\partial K}h_K\|[(\bm{K}^{-1}\bm{w}_h)\cdot\bm{t}_F]\|_F^2+h_K\|[p_h]\|_F^2\big\},
\end{align*}
where $\bm{t}_F$ is a unit tangent vector to $F$.
The next theorem presents a posteriori error estimates of the semi-discrete method \eqref{semidis}.
\begin{theorem}\label{semiresult}
When $d=2$ or $3$, there exists a constant $C_{\text{rel}}$ dependent only on $\mu$, $\alpha$, $\beta$, $\bm{K}$, $\Omega$, $\Gamma_1$ and the shape regularity of $\Th,$ such that
\begin{equation*}
    \begin{aligned}
    &\vertiii{(e_u,e_p,e_w)(t)}^2+\|(e_u,e_p,e_w)\|^2_{L^2(0,t;X)}\leq C_{\text{rel}}\big\{\vertiii{(e_u,e_p,e_w)(0)}^2\\
    &+\left(\int_0^t\mathcal{E}^2_{\Th}(\partial_t\bm{u}_{h},\partial_tp_{h},\bm{w}_h,g)^\frac{1}{2}ds\right)^2+\int_0^t\big(\mathcal{E}^1_{\Th}(\bm{u}_h,p_h,\bm{f})+\mathcal{E}^1_{\Th}(\partial_t\bm{u}_{h},\partial_tp_{h},\partial_t\bm{f})\\
    &\quad+\mathcal{E}^2_{\Th}(\partial_t\bm{u}_{h},\partial_tp_{h},\bm{w}_h,g)+\mathcal{E}^3_{\Th}(p_h,\bm{w}_h)+\mathcal{E}^3_{\Th}(\partial_tp_{h},\partial_t\bm{w}_{h})\big)ds\big\}.
    \end{aligned}
\end{equation*}
\end{theorem}
\begin{proof}
We focus on the case $d=3$ since the proof when $d=2$ is similar. In the proof, we use $C$ to denote generic constant dependent only on $\mu,$ $\alpha,$ $\beta$, $\bm{K}$, $\widetilde{C}_{\text{shape}}$, and $\Omega,$ $\Gamma_1.$ In view of Corollary \ref{cor}, it remains to estimate the norm of each residual. Let $I_h: \bm{V}\rightarrow\bm{V}_h$ denote the Cl\'ement interpolation (see \cite{Clement,Verfurth2013}). Thanks to \eqref{semidis1}, it holds that for each $\bm{v}\in\bm{V},$
\begin{equation}\label{r1}
    \ab{\bm{r}_1}{\bm{v}}=(\bm{f},\bm{v}-I_h\bm{v})-a(\bm{u}_h,\bm{v}-I_h\bm{v})+b(\bm{v}-I_h\bm{v},p_h).
\end{equation}
Element-wise integration by parts leads to
\begin{equation*}
    \begin{aligned}
    \ab{\bm{r}_1}{\bm{v}}&=(\bm{f},\bm{v}-I_h\bm{v})+\sum_{K\in\Th}\left\{-\int_K\bm{\sigma}(\bm{u}_h):\bm{\varepsilon}(\bm{v}-I_h\bm{v})+\int_K\divg(\alpha(\bm{v}-I_h\bm{v}))p_h\right\}\\
    &=\sum_{K\in\Th}\int_K(\bm{f}+\divg\bm{\sigma}(\bm{u}_h)-\alpha\nabla p_h)\cdot(\bm{v}-I_h\bm{v})\\
    &\quad+\sum_{F\in\mathcal{F}(\Th)}\int_F\big([-\bm{\sigma}(\bm{u}_h)+\alpha p_h\bm{I}]\bm{n}_F\big)\cdot(\bm{v}-I_h\bm{v}).
    \end{aligned}
\end{equation*}
Combining the previous equation with the Cauchy--Schwarz inequality and shape-regularity of $\Th$, we obtain
\begin{equation}\label{inter1}
    \ab{\bm{r}_1}{\bm{v}}\leq C\mathcal{E}^1_{\Th}(\bm{u}_h,p_h,\bm{f})^\frac{1}{2}\big(\sum_{K\in\Th}h_K^{-2}\|\bm{v}-I_h\bm{v}\|_K^2+h_K^{-1}\|\bm{v}-I_h\bm{v}\|_{\partial K}^2\big)^\frac{1}{2}.
\end{equation}
It then follows from \eqref{inter1}, the well-known approximation result
\begin{equation*}
    \sum_{K\in\Th}h_K^{-2}\|\bm{v}-I_h\bm{v}\|_K^2+h_K^{-1}\|\bm{v}-I_h\bm{v}\|_{\partial K}^2\leq C|\bm{v}|^2_{H^1(\Omega)},
\end{equation*}
and the Korn's inequality  (cf.~\cite{KO1989})
\begin{equation}\label{Korn}
    |\bm{v}|_{H^1(\Omega)}\leq C\|\bm{v}\|_a,\quad\forall\bm{v}\in\bm{V},
\end{equation}
that
\begin{equation}\label{bd1}
    \|\bm{r}_1\|_{\prime}\leq C\mathcal{E}^1_{\Th}(\bm{u}_h,p_h,\bm{f})^\frac{1}{2}.
\end{equation}
Similarly \eqref{r1} implies that for each $\bm{v}\in\bm{V},$
\begin{equation*}
    \ab{\partial_t\bm{r}_1}{\bm{v}}=(\partial_t\bm{f},\bm{v}-I_h\bm{v})-a(\partial_t\bm{u}_h,\bm{v}-I_h\bm{v})+b(\bm{v}-I_h\bm{v},\partial_tp_h).
\end{equation*}
Then the next estimate
\begin{equation}\label{bd2}
    \|\partial_t\bm{r}_1\|_\prime^2\leq C\mathcal{E}^1_{\Th}(\partial_t\bm{u}_{h},\partial_tp_{h},\partial_t\bm{f}).
\end{equation}
can be proved in the same way as \eqref{bd1}. The norm of $\bm{r}_2$ is trivially estimated by
\begin{equation}\label{bd3}
    \|{r}_2\|_{\prime}\leq C\|g-\beta \partial_tp_{h}-\divg\partial_t\bm{u}_{h}-\divg\bm{w}_h\|.
\end{equation}
To estimate $\|\bm{r}_3\|_\prime$, we use \eqref{semidis3} to obtain
\begin{equation}\label{r3ortho}
    \ab{\bm{r}_3}{\bm{z}}=-e(\bm{w}_h,\bm{z}-\Pi_h\bm{z})+d(\bm{z}-\Pi_h\bm{z},p_h).
\end{equation}
Due to Theorem \ref{regqi}, there exists $\bm{\varphi}\in\bm{V}$ and $\bm{\phi}\in\bm{V}$ such that
\begin{equation}\label{decomp}
    \bm{z}-\Pi_h\bm{z}=\curl(\bm{\varphi}-J_h\bm{\varphi})+\bm{\phi}-\Pi_h\bm{\phi},
\end{equation}
where $\bm{\varphi}$ and $\bm{\phi}$ satisfy \eqref{interperror}.
Using \eqref{r3ortho}, \eqref{decomp}, and element-wise integration by parts, we arrive at
\begin{equation*}
\begin{aligned}
    &\ab{\bm{r}_3}{\bm{z}}=-(\bm{K}^{-1}\bm{w}_h,\curl(\bm{\varphi}-J_h\bm{\varphi}))-(\bm{K}^{-1}\bm{w}_h,\bm{\phi}-\Pi_h\bm{\phi})+(\divg(\bm{\phi}-\Pi_h\bm{\phi}),p_h)\\
    &\quad=\sum_{K\in\Th}\left\{-\int_K\curl(\bm{K}^{-1}\bm{w}_h)\cdot(\bm{\varphi}-J_h\bm{\varphi})-\int_K(\bm{K}^{-1}\bm{w}_h+\nabla p_h)\cdot(\bm{\phi}-\Pi_h\bm{\phi})\right\}\\
    &\qquad+\sum_{F\in\mathcal{F}(\Th)}\left\{\int_F-[(\bm{K}^{-1}\bm{w}_h)\times\bm{n}_F]\cdot(\bm{\varphi}-J_h\bm{\varphi})+\int_F[p_h](\bm{\phi}-\Pi_h\bm{\phi})\cdot\bm{n}_F\right\}.
\end{aligned}
\end{equation*}
It then follows from the previous equation, the Cauchy--Schwarz inequality, and \eqref{interperror} that
\begin{equation}\label{bd4}
    \ab{\bm{r}_3}{\bm{z}}\leq C\mathcal{E}^3_{\Th}(p_h,\bm{w}_h)^\frac{1}{2}\|\bm{z}\|_{\bm{W}}.
\end{equation}
Similarly, it holds that
\begin{equation}\label{bd5}
    \|\partial_t\bm{r}_3\|_\prime\leq C\mathcal{E}^3_{\Th}(\partial_tp_h,\partial_t\bm{w}_h)^\frac{1}{2}.
\end{equation}
Combining \eqref{bd1}--\eqref{bd3}, \eqref{bd4}, \eqref{bd5} completes the proof.
\end{proof}

\section{Fully discrete method}\label{secfully}
Let $0=t_0<t_1<\cdots<t_{N-1}<t_N=T$ and $\tau_n=t_n-t_{n-1}$ for $n=1,2,\ldots,N.$ Let $\Th^n$ be a conforming simplicial triangulation of $\Omega$ aligned with $\Gamma_1$ and $\Gamma_2$. Let  $\bm{V}_h^n$, $Q_h^n$, $\bm{W}_h^n$ be finite element subspaces of $\bm{V}$, $Q,$ $\bm{W}$ described in Section \ref{secsemi} based on grid $\Th^n$, respectively.  We assume that $\{\Th^n\}_{n=0}^N$ is uniformly shape-regular w.r.t.~$n,$ that is,
$$\max_{0\leq n\leq N}\max_{K\in\Th^n}\frac{r_K}{\rho_K}:=C_{\text{shape}}<\infty.$$
Given a sequence $\{\chi^n\}_{n=0}^N$, we define the backward difference as
$$\delta_t\chi^n=\frac{\chi^n-\chi^{n-1}}{\tau_n},$$
and the continuous linear interpolant $\chi^\tau$ on $[0,T]$ as
\begin{align*}
    \chi^\tau(t)=\frac{t-t_{n-1}}{\tau_n}\chi^n+\frac{t_n-t}{\tau_n}\chi^{n-1},\quad t\in[t_{n-1},t_n].
\end{align*}
Notice that $\partial_t\chi^\tau=\delta_t\chi^n$ over $[t_{n-1},t_n]$.
Let $\bm{u}_h^0$, $p_h^0$, $\bm{w}_h^0$ be suitable approximation to $\bm{u}(0)$, $p(0)$, $\bm{w}(0)$, and $\bm{f}^n=\bm{f}(t_n)$, $g^n=g(t_n).$ The fully discrete scheme for \eqref{var} is to find $\bm{u}_h^n\in\bm{V}^n_h$, $p_h^n\in Q^n_h$, $\bm{w}_h^n\in\bm{W}^n_h$ with $n=1,2,\ldots,N$, such that
\begin{subequations}\label{fully}
\begin{align}
    a({\bm{u}}_h^n,v)-b(\bm{v}, {p}_h^n)&=({\bm{f}}^n,\bm{v}),\quad\bm{v}\in\bm{V}^n_h,\label{fully1}\\
    c(\delta_tp_h^n,q)+b(\delta_t\bm{u}_h^n,q)+d({\bm{w}}^n_h,q)&=({g}^n,q),\quad q\in Q^n_h,\\
    e({\bm{w}}^n_h,\bm{z})-d(\bm{z},{p}^n_h)&=0,\quad \bm{z}\in\bm{W}^n_h.
\end{align}
\end{subequations}
Here, \eqref{fully} is derived using the implicit Euler discretization
in time. To apply Lemma \ref{energy}, let
$\bm{u}^\tau_{h}(t)$, $p^\tau_{h}(t)$, $\bm{w}^\tau_{h}(t)$ be the
continuous linear interpolants of $\bm{u}_h^n$, $p_h^n,$ $\bm{w}_h^n$
defined above, respectively.  Let
\begin{equation}\label{error2}
    E_u=\bm{u}-\bm{u}^\tau_{h},\quad E_p=p-p^\tau_{h},\quad E_w=\bm{w}-\bm{w}^\tau_{h}.
\end{equation}
Rewriting \eqref{var} gives the following the error equation
\begin{equation}\label{fullyerror}
    \begin{aligned}
    a(E_u,\bm{v})-b(\bm{v}, E_p)&=\ab{\bm{R}_1}{\bm{v}},\quad\bm{v}\in\bm{V},\\
    c(\partial_tE_p,q)+b(\partial_tE_u,q)+d(E_w,q)&=\ab{R_2}{q},\quad q\in Q,\\
    e(E_w,\bm{z})-d(\bm{z},E_p)&=\ab{\bm{R}_3}{\bm{z}},\quad\bm{z}\in\bm{W},
\end{aligned}
\end{equation}
where residuals $\bm{R}_1(t)\in\bm{V}^\prime$, $\bm{R}_2(t)\in Q^\prime$, $\bm{R}_3(t)\in\bm{W}^\prime$ for each $t\in[0,T]$ are
\begin{equation}\label{R123}
    \begin{aligned}
    \ab{\bm{R}_1}{\bm{v}}&:=(\bm{f},\bm{v})-a(\bm{u}^\tau_{h},\bm{v})+b(\bm{v},p^\tau_{h}),\\
    \ab{R_2}{q}&:=(g,q)-c(\delta_tp_h^n,q)-b(\delta_t\bm{u}_h^n,q)-d(\bm{w}^\tau_{h},q),\\
    \ab{\bm{R}_3}{\bm{z}}&:=-e(\bm{w}^\tau_{h},\bm{z})+d(\bm{z},p^\tau_{h}).
\end{aligned}
\end{equation}

Similarly to the error analysis of the semi-discrete problem, it suffices to analyze the dual norm of the residuals, defined in~\eqref{R123}. However, the mesh used in the fully discrete scheme is allowed to change at different time levels and we introduce several useful operations in the following. Given two triangulations $\CT_{h_1}$ and $\CT_{h_2}$ of $\Omega$, let  $\CT_{h_1}\vee \CT_{h_2}$ denote the minimal common refinement, i.e., $\CT_{h_1}\vee \CT_{h_2}$ is the coarsest conforming triangulation that is a refinement of both  $\CT_{h_1}$ and $\CT_{h_2}$. Similarly, let  $\CT_{h_1}\wedge \CT_{h_2}$ denote the maximal common coarsening, i.e., $\CT_{h_1}\wedge \CT_{h_2}$ is the finest conforming triangulation that is a coarsening of both $\CT_{h_1}$ and $\CT_{h_2}$. The operations $\wedge$ and $\vee$ on triangulations are widely used in adaptivity literature, see, e.g., \cite{CKNS2008,DKS2016}.

To handle simultaneously  $(\bm{u}_h^n,p_h^n,\bm{w}_h^n)$ and $(\bm{u}_h^{n-1},p_h^{n-1},\bm{w}_h^{n-1}),$ we assume that for $1\leq n\leq N$ and consecutive meshes $\Th^n$ and  $\Th^{n-1}$,  the maximal common coarsening $\Th^n\wedge\Th^{n-1}$ and the minimal common refinement $\Th^n\vee\Th^{n-1}$ exist. As is well known, this assumption is true when $\{\Th^n\}_{n=0}^N$ are newest vertex bisection refinement of the same macrotriangulation, cf.~\cite{Lakkis2006,CKNS2008}. In addition, there is a uniform bound on the ratio of the sizes of elements in $K\in\Th^n\vee\Th^{n-1}$ and of elements $K^\prime\in\Th^n\wedge\Th^{n-1}$ contained in $K,$ that is,
\begin{equation}\label{samesize}
    \sup_{1\leq n\leq N}\sup_{K^\prime\subset K, K^\prime\in\Th^n\vee\Th^{n-1}}\sup_{K\in\Th^n\wedge\Th^{n-1}}\frac{h_K}{h_{K^\prime}}:= C_{\text{ratio}}<\infty.
\end{equation}
Similar assumptions are made in a posteriori error estimation of the heat equation, see, e.g., \cite{Verfurth2003}. Let $\bm{f}_h^n\in\bm{V}_h^n$ and $g_h^n\in Q_h^n$ be approximations to $\bm{f}^n$ and $g^n$, respectively. Within the interval $[t_{n-1},t_n]$, we split the residuals into
\begin{subequations}\label{split}
\begin{align}
     \bm{R}_1&=\bm{f}-\bm{f}_h^n+\bm{S}^n_1+\bm{T}^n_1,\label{split1}\\
    R_2&=g-g_h^n+S^n_2+T^n_2,\label{split2}\\
    \bm{R}_3&=\bm{S}^n_3+\bm{T}^n_3,\label{split3}
\end{align}
\end{subequations}
where the spatial residuals $\bm{S}^n_1\in\bm{V}^\prime$, $\bm{S}^n_2\in Q^\prime$, $\bm{S}^n_3\in\bm{W}^\prime$ are defined as
\begin{align*}
    \ab{\bm{S}^n_1}{\bm{v}}&:=\ab{\bm{f}_h^n}{\bm{v}}-a(\bm{u}^n_{h},\bm{v})+b(\bm{v},p^n_{h}),\\
    \ab{S^n_2}{q}&:=\ab{g_h^n}{q}-c(\delta_tp_h^n,q)-b(\delta_t\bm{u}_h^n,q)-d(\bm{w}^n_{h},q),\\
    \ab{\bm{S}^n_3}{\bm{z}}&:=-e(\bm{w}^n_{h},\bm{z})+d(\bm{z},p^n_{h}),
\end{align*}
and the temporal residuals $\bm{T}^n_1(t)\in\bm{V}^\prime$, $\bm{T}^n_2(t)\in Q^\prime$, $\bm{T}^n_3(t)\in\bm{W}^\prime$ for $t\in[t_{n-1},t_n]$ are
\begin{align*}
    \ab{\bm{T}^n_1}{\bm{v}}&:=a(\bm{u}^n_{h}-\bm{u}_h^\tau,\bm{v})-b(\bm{v},p^n_{h}-p_h^\tau),\\
    \ab{T^n_2}{q}&:=d(\bm{w}^n_{h}-\bm{w}^\tau_{h},q),\\
    \ab{\bm{T}^n_3}{\bm{z}}&:=e(\bm{w}^n_{h}-\bm{w}^\tau_{h},\bm{z})-d(\bm{z},p^n_{h}-p^\tau_{h}).
\end{align*}
By \eqref{R123}, the temporal derivatives of $\bm{R}_1$ and $\bm{R}_3$ over $[t_{n-1},t_n]$ are
\begin{subequations}
\begin{align}
\ab{\partial_t\bm{R}_1}{\bm{v}}&=\ab{\partial_t\bm{f}-\delta_t\bm{f}_h^n}{\bm{v}}+\ab{\delta_t\bm{S}^n_1}{\bm{v}},\label{splitdtR1}\\
\ab{\partial_t\bm{R}_3}{\bm{v}}&=\ab{\delta_t\bm{S}^n_3}{\bm{v}}.\label{splitdtR3}
\end{align}
\end{subequations}
Next, we use the following fully discrete spatial error indicators which can be viewed as fully discrete counterparts of the error estimators introduced in Section \ref{secsemi}. We set,
\begin{align*}
    &\mathcal{E}_1^n:=\mathcal{E}^1_{\Th^n}({\bm{u}}^n_h,{p}^n_h,\bm{f}_h^n),\quad\mathcal{E}_{1,t}^n:=\mathcal{E}^1_{\Th^n\vee\Th^{n-1}}(\delta_t{\bm{u}}^n_h,\delta_t{p}^n_h,\delta_t\bm{f}_h^n),\\
    &\mathcal{E}_2^n:=\mathcal{E}^2_{\Th^n}(\delta_t\bm{u}_h^n,\delta_tp_h^n,\bm{w}_h^n,g^n_h),\\
    &\mathcal{E}_3^n:=\mathcal{E}^3_{\Th^n}({p}^n_h,{\bm{w}}^n_h),\quad\mathcal{E}_{3,t}^n:=\mathcal{E}^3_{\Th^n\vee\Th^{n-1}}(\delta_tp^n_h,\delta_t{\bm{w}}^n_h).
\end{align*}
Throughout the rest of the presentation, we shall write $A\lesssim B$ provided
$A\leq CB$, where $C$ is a constant depending only on $\mu$, $\alpha$,
$\beta$, $\bm{K}$, $\Omega$, $\Gamma_1,$ $C_{\text{shape}}$,
$C_{\text{ratio}}$. Since the spatial residuals are time-independent, their norms can be estimated as in the proof of Theorem~\ref{semiresult}. We have the following result.
\begin{lemma}\label{spacelemma}
For $1\leq n\leq N$, it holds that
\begin{subequations}
\begin{align}
    \|\bm{S}^n_1\|^2_\prime&\lesssim\mathcal{E}_1^n,\label{space1}\\
    \|\delta_t\bm{S}^n_1\|^2_\prime&\lesssim\mathcal{E}_{1,t}^n,\label{space2}\\
    \|\bm{S}^n_2\|^2_\prime&\lesssim\mathcal{E}_2^n,\\
    \|\bm{S}^n_3\|^2_\prime &\lesssim\mathcal{E}_3^n,\\
    \|\delta_t\bm{S}^n_3\|^2_\prime&\lesssim\mathcal{E}_{3,t}^n.
\end{align}
\end{subequations}
\end{lemma}
\begin{proof}
For $\bm{v}\in\bm{V}$, let $\bm{v}_h$ be the Cl\'ement interpolant on $\Th^n$. It follows from \eqref{fully1} and element-wise integration by parts that
\begin{equation}\label{repS1}
    \begin{aligned}
    \ab{\bm{S}^n_1}{\bm{v}}&=\ab{\bm{S}^n_1}{\bm{v}-\bm{v}_h}=\sum_{K\in\Th}\int_K(\bm{f}_h^n+\divg\bm{\sigma}(\bm{u}_h)-\alpha\nabla p_h)\cdot(\bm{v}-\bm{v}_h)\\
    &-\sum_{F\in\mathcal{F}(\Th)}\int_F\big([\bm{\sigma}(\bm{u}_h)-\alpha p_h\bm{I}]\bm{n}_F\big)\cdot(\bm{v}-\bm{v}_h).
    \end{aligned}
\end{equation}
Using \eqref{repS1} and the same analysis of estimating $\|\bm{r}_1\|_\prime$ in Theorem \ref{semiresult}, we obtain
\begin{equation*}
    \|\bm{S}^n_1\|_\prime=\sup_{\bm{v}\in\bm{V},\|\bm{v}\|_a=1}\ab{\bm{S}^n_1}{\bm{v}}\lesssim(\mathcal{E}_1^n)^\frac{1}{2}.
\end{equation*}
For any $\bm{v}\in\bm{V}$, let $\tilde{\bm{v}}_h$ be the Cl\'ement interpolant of $\bm{v}$ on $\Th^n\wedge\Th^{n-1}$. Using \eqref{samesize}, \eqref{fully1}, and integrating by parts over $\Th^n\vee\Th^{n-1}$, and following again the same analysis of estimating $\|\bm{r}_1\|_\prime$ in Theorem \ref{semiresult}, we obtain a similar estimate:
\begin{equation*}
\begin{aligned}
        &\|\delta_t\bm{S}^n_1\|_\prime=\sup_{\bm{v}\in\bm{V},\|\bm{v}\|_a=1}\ab{\delta_t\bm{S}^n_1}{\bm{v}}=\sup_{\bm{v}\in\bm{V},\|\bm{v}\|_a=1}\ab{\delta_t\bm{S}^n_1}{\bm{v}-\tilde{\bm{v}}_h}\\
        &=\sup_{\bm{v}\in\bm{V},\|\bm{v}\|_a=1}(\delta_t\bm{f}_h^n,\bm{v}-\tilde{\bm{v}}_h)-a(\delta_t\bm{u}^n_h,\bm{v}-\tilde{\bm{v}}_h)+b(\bm{v}-\tilde{\bm{v}}_h,\delta_tp^n_h)\lesssim(\mathcal{E}_{1,t}^n)^\frac{1}{2}.
\end{aligned}
\end{equation*}
The remaining estimates can be proved in the same way.
\end{proof}

Let $\|\bm{f}-\bm{f}_h^i\|_{\prime}=\|\bm{f}-\bm{f}_h^i\|_{\bm{V}^\prime}$ and $\|\partial_t\bm{f}-\delta_t\bm{f}_h^i\|_{\prime}=\|\partial_t\bm{f}-\delta_t\bm{f}_h^i\|_{\bm{V}^\prime}$.
We present the first main result of this paper in the following theorem.
\begin{theorem}\label{fullyresult}
For $1\leq i\leq n$, let $\bm{f}_h^i$ be the $L^2$-projection of $\bm{f}^i$ onto $\bm{V}_h^i$, $g_h^i$ the $L^2$-projection of $g^i$ onto $Q_h^i$. There exists a constant $C_{\text{drel}}$ dependent only on $\mu$, $\alpha$, $\beta$, $\bm{K}$, $\Omega,$ $C_{\text{shape}}, C_{\text{ratio}}$ such that for $n=1,2,\ldots, N,$ the error defined in \eqref{error2} satisfies
\begin{equation*}
    \begin{aligned}
    &\vertiii{(E_u,E_p,E_w)(t_n)}^2+\|(E_u,E_p,E_w)\|_{L^2(0,t_n;X)}^2\\
    &\quad\leq C_{\text{drel}}\big\{\vertiii{(E_u,E_p,E_w)(0)}^2+\big(\sum_{i=1}^n\tau_i\widetilde{\mathcal{E}}^i_{\text{time}}+\tau_i(\mathcal{E}_2^i)^\frac{1}{2}+\widetilde{\mathcal{E}}^i_{\text{data}}\big)^2\\
    &\qquad+\sum_{i=1}^n\tau_i\mathcal{E}^i_{\text{time}}+\tau_i\mathcal{E}^i_{\text{space}}+\mathcal{E}^i_{\text{data}}\},
    \end{aligned}
\end{equation*}
where
\begin{align*}
\widetilde{\mathcal{E}}^i_{\text{time}}&=\|\divg(\bm{w}_h^i-\bm{w}_h^{i-1})\|,\quad\widetilde{\mathcal{E}}^i_{\text{data}}=\int_{t_{i-1}}^{t_i}\|g-g_h^i\|dt,\\
    \mathcal{E}^i_{\text{time}}&=\|\bm{u}_h^i-\bm{u}_h^{i-1}\|^2_a+\|p_h^i-p_h^{i-1}\|^2_c+\|\bm{w}_h^i-\bm{w}_h^{i-1}\|^2_{\bm{W}},\\
    \mathcal{E}_{\text{space}}^i&=\mathcal{E}_1^i+\mathcal{E}_{1,t}^i+\mathcal{E}_2^i+\mathcal{E}_3^i+\mathcal{E}_{3,t}^i,\\
    \mathcal{E}^i_{\text{data}}&=\int_{t_{i-1}}^{t_i}\big(\|\bm{f}-\bm{f}_h^i\|^2_{\prime}+\|\partial_t\bm{f}-\delta_t\bm{f}_h^i\|_{\prime}^2+\|g-g_h^i\|^2\big)dt.
\end{align*}
\end{theorem}
\begin{proof}
Applying Lemma \ref{energy} to \eqref{fullyerror} yields
\begin{equation*}
    \begin{aligned}
    &\vertiii{(E_u,E_p,E_w)(t_n)}^2+\|(E_u,E_p,E_w)\|_{L^2(0,t_n;X)}^2\\
    &\leq C_{\text{stab}}\big(\vertiii{(E_u,E_p,E_w)(0)}^2+\left(\sum_{i=1}^n\int_{t_{i-1}}^{t_i}\|R_2\|_\prime dt\right)^2\\
    &+\sum_{i=1}^n\int_{t_{i-1}}^{t_i}\|\bm{R}_1\|_\prime^2+\|\partial_t\bm{R}_1\|_\prime^2+\|R_2\|_\prime^2+\|\bm{R}_3\|_\prime^2+\|\partial_t\bm{R}_3\|_\prime^2\big)dt.
    \end{aligned}
\end{equation*}
For any $\bm{v}\in\bm{V}$,
the continuity of $a(\cdot,\cdot)$ and $b(\cdot,\cdot)$ implies that
\begin{equation}\label{parta}
    \|\bm{T}^i_1\|_\prime=\sup_{\bm{v}\in\bm{V},\|\bm{v}\|_a=1}\ab{\bm{T}^i_1}{\bm{v}}\lesssim\|{\bm{u}}^i_h-\bm{u}^\tau_{h}\|_a+\|{p}^i_h-p^\tau_{h}\|_c.
\end{equation}
A combination of \eqref{split1}, \eqref{space1} and \eqref{parta} then shows that when $t\in[t_{i-1},t_i]$,
\begin{equation}\label{ineterR1bd}
    \begin{aligned}
    \|\bm{R}_1\|_\prime&\lesssim\|\bm{f}-\bm{f}_h^i\|_{\prime}+\|{\bm{u}}^i_h-\bm{u}^\tau_{h}\|_a+\|{p}^i_h-p^\tau_{h}\|_c+(\mathcal{E}_1^i)^\frac{1}{2}.
    \end{aligned}
\end{equation}
For $t\in[t_{i-1},t_i]$, it is readily checked that
\begin{equation}\label{uttau}
\begin{aligned}
    \|{\bm{u}}^i_h-\bm{u}^\tau_{h}\|_a&=\frac{t_i-t}{\tau_i}\|{\bm{u}}^i_h-\bm{u}^{i-1}_{h}\|_a,\\
    \|p^i_h-p^\tau_{h}\|_c&=\frac{t_i-t}{\tau_i}\|p^i_h-p^{i-1}_{h}\|_c.
\end{aligned}
\end{equation}
Integrating \eqref{ineterR1bd} over $[t_{i-1},t_i]$ and using \eqref{uttau}, we obtain
\begin{equation}\label{R1}
    \int_{t_{i-1}}^{t_i}\|\bm{R}_1\|^2_\prime dt\lesssim\int_{t_{i-1}}^{t_i}\|\bm{f}-\bm{f}_h^i\|^2_{\prime}dt+\tau_i\|\bm{u}^i_h-\bm{u}^{i-1}_{h}\|^2_a+\tau_i\|{p}^i_h-p^{i-1}_{h}\|^2_c+\tau_i\mathcal{E}_1^i.
\end{equation}
On the other hand,
using \eqref{splitdtR1} and \eqref{space2}, we obtain for $t\in[t_{i-1},t_i]$,
\begin{equation}\label{dtR1}
    \|\partial_t\bm{R}_1\|_\prime\lesssim\|\partial_t\bm{f}-\delta_t\bm{f}_h^i\|_{\prime}+(\mathcal{E}_{1,t}^i)^\frac{1}{2}.
\end{equation}
Similarly, using \eqref{split}, \eqref{splitdtR3}, Lemma \ref{spacelemma}, and \eqref{uttau}, one can estimate $\bm{R}_3$, $\partial_t\bm{R}_3$, $R_2$ and obtain the following bounds
\begin{equation}\label{R3}
    \begin{aligned}
\int_{t_{i-1}}^{t_i}\|\bm{R}_3\|^2_\prime dt&\lesssim\tau_i\|\bm{w}^i_h-\bm{w}_h^{i-1}\|^2_{\bm{W}}+\tau_i\|p^i_h-p^{i-1}_{h}\|^2_c+\tau_i\mathcal{E}_3^i,\\
\int_{t_{i-1}}^{t_i}\|\partial_t\bm{R}_3\|^2_\prime dt&\lesssim\tau_i\mathcal{E}_{3,t}^i,
\end{aligned}
\end{equation}
and
\begin{equation}\label{R2}
    \begin{aligned}
    \int_{t_{i-1}}^{t_i}\|R_2\|_\prime dt&\lesssim\int_{t_{i-1}}^{t_i}\big\|g-g_h^i\|dt+\tau_i\|\divg(\bm{w}^{i-1}_{h}-\bm{w}_h^i)\|+\tau_i(\mathcal{E}_2^i)^\frac{1}{2},\\
    \int_{t_{i-1}}^{t_i}\|R_2\|^2_\prime dt&\lesssim\int_{t_{i-1}}^{t_i}\big\|g-g_h^i\|^2dt+\tau_i\|\divg(\bm{w}^{i-1}_{h}-\bm{w}_h^i)\|^2+\tau_i\mathcal{E}_2^i.
\end{aligned}
\end{equation}
Combining \eqref{R1}--\eqref{R2} completes the proof.
\end{proof}
\begin{remark}
The first two terms in $\mathcal{E}^i_{\text{data}}$ can be further estimated by
\begin{align*}
    \|\bm{f}-\bm{f}_h^i\|_{\prime}&\lesssim\sum_{K\in \Th^i}h_K^2\|\bm{f}-\bm{f}_h^i\|^2_K,\\
    \|\partial_t\bm{f}-\delta_t\bm{f}_h^i\|_{\prime}&\lesssim\sum_{K\in \Th^i\wedge\Th^{i-1}}h_K^2\|\partial_t\bm{f}-\delta_t\bm{f}_h^i\|^2_K.
\end{align*}
\end{remark}

\section{Lower bound}\label{seclower}
In this section, we show that $\tau_n\mathcal{E}^n_{\text{time}}$ and $\tau_n\mathcal{E}^n_{\text{space}}$ are lower bounds of the space-time discretization error of the fully discrete scheme \eqref{fully}. First we present a lemma comparing the spatial residual with the spatial error indicators. Since the spatial error estimators and residuals are time-independent, the proof follows from the well-known Verf\"urth bubble function technique for a posteriori error estimates for stationary Stokes and Poisson's equations, see, e.g., \cite{Verfurth1991,Alonso1996,DH2014}. Throughout the rest, $\underline{C}$ is a generic constant that  depends only on $\lambda,$ $\mu$, $\alpha,$ $\beta,$ $\bm{K}$, $\Omega$, $\Gamma_1,$ $C_{\text{shape}}$, $C_{\text{ratio}}$. Hence the constant in some lower bounds may not be locking free.
\begin{lemma}\label{lowerspace}
Let $\lambda$, $\mu$, $\bm{K}$ be piecewise constants over $\Th^n$. For $1\leq n\leq N$, it holds that
\begin{subequations}
\begin{align}
    \mathcal{E}_1^n&\leq \underline{C}\|\bm{S}^n_1\|_{\prime}^2,\label{lowerS1}\\
    \mathcal{E}_{1,t}^n&\leq \underline{C}\|\delta_t\bm{S}^n_1\|^2_{\prime},\label{lowerS2}\\
    \mathcal{E}_2^n&\lesssim\|S_2^n\|_{\prime}^2,\\
    \mathcal{E}_3^n&\lesssim\|\bm{S}^n_3\|_\prime^2,\\
    \mathcal{E}_{3,t}^n&\lesssim \|\delta_t\bm{S}^n_3\|^2_\prime.
\end{align}
\end{subequations}
\end{lemma}
\begin{proof}
To prove \eqref{lowerS1}, it suffices to find $\bm{v}\in\bm{V}$, such that
\begin{equation*}
    \mathcal{E}_1^n\leq \underline{C}\ab{\bm{S}_1^n}{\bm{v}},\quad\|\bm{v}\|^2_a\leq \underline{C}\mathcal{E}_1^n.
\end{equation*}
For each $K\in\Th$, let $R_K=(\bm{f}_h^n+\divg\bm{\sigma}(\bm{u}_h)-\alpha\nabla p_h)|_K$ and let $\phi_K$ denote the volume bubble function supported on $K$ so that the maximum is $1.$. For each $F\in\mathcal{F}(\Th)$, let $J_F=-[\bm{\sigma}(\bm{u}_h)-\alpha p_h\bm{I}]|_F\bm{n}_F$, and recall that $\phi_F$ is the face bubble function supported on union of neighboring elements of $F$. The desired $\bm{v}$ is then defined as
\begin{equation*}
    \bm{v}=\gamma_1\sum_{K\in\Th}h_K^2R_K\phi_K+\gamma_2\sum_{F\in\mathcal{F}(\Th)}h_FJ_F\phi_F,
\end{equation*}
where $h_F$ is the diameter of $F,$ and $\gamma_1, \gamma_2$ are undetermined constants. Using the Cauchy--Schwarz inequality and finite overlapping of supports of $\{\phi_K\}$ and $\{\phi_F\}$, one case easily show that $\|\bm{v}\|^2_a\leq \underline{C}\mathcal{E}_{1}^n$. On the other hand, \eqref{repS1} implies
\begin{equation*}
    \ab{\bm{S}_1^n}{\bm{v}}=\sum_{K\in\Th}\int_KR_K\cdot\bm{v}+\sum_{F\in\mathcal{F}(\Th)}\int_FJ_F\cdot\bm{v}.
\end{equation*}
$\ab{\bm{S}_1^n}{\bm{v}}\geq \underline{C}\mathcal{E}_1^n$ then
follows from Young's inequality and suitablly chosen $\gamma_1$,
$\gamma_2$, see, e.g., Lemma~5.1 in \cite{Verfurth2003} for details. Other
lower bounds can be shown in an analogous fashion.
\end{proof}
\begin{remark}
Based on the $\|\cdot\|_a$-norm, it seems that the dependence on $\lambda$ in lower bounds \eqref{lowerS1} and \eqref{lowerS2} cannot be avoided. To obtain an error estimator that is a \emph{robust} lower bound, one can apply the analysis here to the four- or five-field formulation \cite{LJJ2016,ARN2019} in Biot's consolidation model.
\end{remark}

We present the second main result in the following theorem. Similar technique in the proof was used in \cite{Verfurth2003} for proving the lower bound in a posteriori error estimation for the primal formulation of the heat equation.
\begin{theorem}\label{lower}
Let $\lambda,$ $\mu$, $\bm{K}$ be piecewise constants on $\Th^n$. For $n=1,2,\ldots, N,$
\begin{align*}
    &\tau_n\mathcal{E}^n_{\text{time}}+\tau_n\mathcal{E}^n_{\text{space}}\leq \underline{C}\big\{\|(E_u,E_p,E_w)\|^2_{L^2(t_{n-1},t_n;X)}\\
    &\qquad+\int_{t_{n-1}}^{t_n}\big(\|\bm{f}-\bm{f}_h^n\|^2_{\prime}+\|\partial_t\bm{f}-\delta_t\bm{f}_h^n\|_{\prime}^2+\|g-g_h^n\|^2\big)dt\big\}.
\end{align*}
\end{theorem}
\begin{proof}
First by \eqref{var} and definitions of residuals in \eqref{R123}, we have
\begin{subequations}\label{Rbound}
\begin{align}
    \|\bm{R}_1\|_{\prime}&\lesssim\|E_u\|_{a}+\|E_p\|_c,\label{R1prime}\\
    \|{R}_2\|_{\prime}&\lesssim\|\partial_tE_p\|_c+\|\partial_t\divg E_u\|+\|\divg E_w\|,\label{R2prime}\\
    \|\bm{R}_3\|_{\prime}&\lesssim\|E_w\|_e+\|E_p\|_c,\\
    \|\partial_t\bm{R}_1\|_{\prime}&\lesssim\|\partial_tE_u\|_{a}+\|\partial_tE_p\|_c,\label{R4prime}\\
    \|\partial_t\bm{R}_3\|_\prime&\lesssim{\|\partial_tE_w\|_{\bm{W}^\prime}}+\|\partial_tE_p\|_c.
\end{align}
\end{subequations}
Consider the bilinear form
$$B(\bm{w},p;\bm{z},q)=e(\bm{w},\bm{z})-d(\bm{z},p)+d(\bm{w},q)$$
of the mixed formulation of the elliptic equation. Due to the inf-sup condition of $B,$
there exist $\bm{z}\in\bm{W}$ and $q\in Q$ with $\|\bm{z}\|_{\bm{W}}=1, \|q\|_c=1$ such that
\begin{align*}
    &\|\bm{w}^n_{h}-\bm{w}^\tau_{h}\|_{\bm{W}}+\|p^n_{h}-p^\tau_{h}\|_c\lesssim B(\bm{w}^n_{h}-\bm{w}^\tau_{h},p^n_{h}-p^\tau_{h};\bm{z},q)\\
    &\qquad=\ab{T^n_2}{q}+\ab{\bm{T}^n_3}{\bm{z}},\quad t\in[t_{n-1},t_n].
\end{align*}
Using the previous estimate, the triangle and Cauchy--Schwarz inequalities, we obtain
\begin{equation}\label{bdwptau}
    \begin{aligned}
       &\|\bm{w}^n_{h}-\bm{w}^\tau_{h}\|_{\bm{W}}+\|p^n_{h}-p^\tau_{h}\|_c\\
       &\lesssim\ab{R_2}{q}+\ab{\bm{R}_3}{\bm{z}}-\ab{g-g_h^n}{q}-\ab{S^n_2}{q}-\ab{\bm{S}^n_3}{\bm{z}}\\
       &\lesssim\|R_2\|_{\prime}+\|\bm{R}_3\|_{\prime}+\|g-g_h^n\|+\|\bm{S}^n_2\|_{\prime}+\|\bm{S}^n_3\|_\prime.
    \end{aligned}
\end{equation}
Setting $\bm{v}=(\bm{u}^n_{h}-\bm{u}^\tau_{h})/\|\bm{u}^n_{h}-\bm{u}^\tau_{h}\|_a$ in the definition of $\bm{T}^n_1$, we have for $t\in[t_{n-1},t_n]$,
\begin{equation}\label{bdutau}
    \begin{aligned}
       &\|\bm{u}^n_{h}-\bm{u}^\tau_{h}\|_a=\ab{\bm{T}^n_1}{\bm{v}}+b(\bm{v},p_h^n-p_h^\tau)\\
       &\quad=\ab{\bm{R}_1}{\bm{v}}-\ab{\bm{f}-\bm{f}_h^n}{\bm{v}}-\ab{\bm{S}_1^n}{\bm{v}}+b(\bm{v},p_h^n-p_h^\tau)\\
       &\quad\leq\|\bm{R}_1\|_\prime+\|\bm{f}-\bm{f}_h^n\|_{\prime}+\|\bm{S}^n_1\|_\prime+\|p_h^n-p_h^\tau\|_c.
    \end{aligned}
\end{equation}
A combination of \eqref{bdwptau}, \eqref{bdutau}, \eqref{uttau} and Lemma \ref{spacelemma} shows that
\begin{equation}\label{timebound1}
\begin{aligned}
        &\tau_n\mathcal{E}^n_{\text{time}}\lesssim\int_{t_{n-1}}^{t_n}\big(\|\bm{f}-\bm{f}_h^n\|^2_{\prime}+\|g-g_h^n\|^2+\|\bm{R}_1\|^2_{\prime}\\
    &\quad+\|R_2\|^2_{\prime}+\|\bm{R}_3\|^2_{\prime}\big)dt+\tau_n\big(\mathcal{E}^n_1+\mathcal{E}^n_2+\mathcal{E}^n_3\big).
\end{aligned}
\end{equation}
It remains to estimate $\mathcal{E}^n_1$, $\mathcal{E}^n_2$, and $\mathcal{E}^n_3$. Let
$$\phi(t)=(\alpha+1)\left(\frac{t-t_{n-1}}{\tau_n}\right)^\alpha,$$
where $\alpha>0$ is a constant that will be specified later. It follows from Lemma \ref{lowerspace} and the triangle inequality that
\begin{equation}\label{lowinter1}
    \begin{aligned}
       &\tau_n\big(\mathcal{E}^n_1+\mathcal{E}^n_2+\mathcal{E}^n_3\big)=\int_{t_{n-1}}^{t_n}\phi(t)\big(\mathcal{E}^n_1+\mathcal{E}^n_2+\mathcal{E}^n_3\big)dt\\
       &\quad\leq \underline{C}\int_{t_{n-1}}^{t_n}\phi(t)\big(\|\bm{S}_1^n\|^2_\prime+\|\bm{S}_2^n\|^2_\prime+\|\bm{S}_3^n\|^2_\prime\big)dt\\
       &\quad\leq \underline{C}(\alpha+1) \int_{t_{n-1}}^{t_n}\big(\|\bm{R}_1\|_{\prime}^2+\|\bm{f}-\bm{f}_h^n\|^2_{\prime}+\|R_2\|_{\prime}^2\\
       &\quad+\|g-g_h^n\|^2_{\prime}+\|\bm{R}_3\|_{\prime}^2\big)dt+\underline{C}\int_{t_{n-1}}^{t_n}\phi(t)\big(\|\bm{T}^n_1\|_{\prime}^2+\|T^n_2\|_{\prime}^2+\|\bm{T}^n_3\|_{\prime}^2\big)dt,
    \end{aligned}
\end{equation}
where the generic constant $\underline{C}$ is  independent of $\alpha.$
For ant $\bm{v}\in\bm{V}$ with $\|\bm{v}\|_a=1$, direct calculation shows that
\begin{equation*}
    \begin{aligned}
       &\int_{t_{n-1}}^{t_n}\phi(t)\ab{\bm{T}^n_1}{\bm{v}}^2dt=\int_{t_{n-1}}^{t_n}\phi(t)\{a(\bm{u}^n_{h}-\bm{u}_h^\tau,\bm{v})-b(\bm{v},p^n_{h}-p_h^\tau)\}^2dt\\
       &\quad=\{a(\bm{u}^n_{h}-\bm{u}_h^{n-1},\bm{v})-b(\bm{v},p^n_{h}-p_h^{n-1})\}^2\int_{t_{n-1}}^{t_n}\phi(t)\left(\frac{t_n-t}{\tau_n}\right)^2dt\\
       &\quad\leq \underline{C}\big(\|\bm{u}^n_{h}-\bm{u}_h^{n-1}\|^2_a+\|p^n_{h}-p_h^{n-1}\|^2_c\big)\tau_n\int_0^1s^\alpha(1-s)^2ds\\
       &\quad\leq\underline{C}\tau_n\mathcal{E}_{\text{time}}^nF(\alpha),
    \end{aligned}
\end{equation*}
where
$F(\alpha)=1-\frac{2(\alpha+1)}{\alpha+2}+\frac{\alpha+1}{\alpha+3}.$ Hence
\begin{equation}\label{timedual1}
\int_{t_{n-1}}^{t_n}\phi(t)\|\bm{T}^n_1\|^2_\prime dt=\sup_{\bm{v}\in\bm{V},\|\bm{v}\|_a=1}\int_{t_{n-1}}^{t_n}\phi(t)\ab{\bm{T}^n_1}{\bm{v}}^2dt\leq \underline{C}\tau_n\mathcal{E}_{\text{time}}^nF(\alpha).
\end{equation}
Similarly, we have
\begin{equation}\label{timedual23}
\int_{t_{n-1}}^{t_n}\phi(t)\big(\|T^n_2\|^2_\prime+\|\bm{T}^n_3\|^2_\prime\big)dt\leq\underline{C}\tau_n\mathcal{E}_{\text{time}}^nF(\alpha).
\end{equation}
Combining \eqref{lowinter1},  \eqref{timedual1} and \eqref{timedual23}, we obtain
\begin{equation}\label{spacebound1}
\begin{aligned}
&\tau_n\big(\mathcal{E}^n_1+\mathcal{E}^n_2+\mathcal{E}^n_3\big)\leq\underline{C}(\alpha+1)\int_{t_{n-1}}^{t_n}\big(\|\bm{R}_1\|_{\prime}^2+\|\bm{f}-\bm{f}_h^n\|^2_{\prime}+\|R_2\|_{\prime}^2\\
&\quad+\|g-g_h^n\|^2+\|\bm{R}_3\|_{\prime}^2\big)dt+\underline{C}\tau_n\mathcal{E}^n_{\text{time}}F(\alpha).
\end{aligned}
\end{equation}
Note that $F(\alpha)\rightarrow0$ as $\alpha\rightarrow\infty.$ It then follows from \eqref{timebound1} and  \eqref{spacebound1} with sufficiently large $\alpha$ that
\begin{equation}\label{timebound2}
\begin{aligned}
        \tau_n\mathcal{E}^n_{\text{time}}\leq \underline{C}\int_{t_{n-1}}^{t_n}\big(\|\bm{f}-\bm{f}_h^n\|^2_{\prime}+\|g-g_h^n\|^2+\|\bm{R}_1\|^2_{\prime}+\|R_2\|^2_{\prime}+\|\bm{R}_3\|^2_{\prime}\big)dt.
\end{aligned}
\end{equation}
Clearly, the bounds \eqref{spacebound1} and \eqref{timebound2} imply that
\begin{equation}\label{spacebound4}
    \begin{aligned}
       &\tau_n\big(\mathcal{E}^n_1+\mathcal{E}^n_2+\mathcal{E}^n_3\big)\\
       &\leq \underline{C}\int_{t_{n-1}}^{t_n}\big(\|\bm{f}-\bm{f}_h^n\|^2_{\prime}+\|g-g_h^n\|^2+\|\bm{R}_1\|^2_{\prime}+\|R_2\|^2_{\prime}+\|\bm{R}_3\|^2_{\prime}\big)dt.
    \end{aligned}
\end{equation}
Finally, using Lemma \ref{lowerspace} and \eqref{splitdtR1}, \eqref{splitdtR3}, we have
\begin{equation}\label{spacebound5}
    \begin{aligned}
       &\mathcal{E}^n_{1,t}+\mathcal{E}^n_{3,t}\leq \underline{C}\big(\|\delta_t\bm{S}^n_1\|^2_{\prime}+\|\delta_t\bm{S}^n_3\|^2_{\prime}\big)\\
       &\quad\leq \underline{C}\big(\|\partial_t\bm{f}-\delta_t\bm{f}_h^n\|_{\prime}^2+\|\partial_t\bm{R}_1\|^2_{\prime}+\|\partial_t\bm{R}_3\|^2_{\prime}\big)
    \end{aligned}
\end{equation}
on $[t_{n-1},t_n].$
Combining \eqref{timebound2}, \eqref{spacebound4},
\eqref{spacebound5}, and \eqref{Rbound} then completes the proof.
\end{proof}

{In practice, one can use $\mathcal{E}^n_{\text{time}}$
  and $\mathcal{E}^n_{\text{data}}$ to adjust the time step size and
  $\mathcal{E}^n_{\text{space}}$ to refine and coarsen the spatial
  mesh. Due to the complexity of space-time adaptivity, we shall not
  present a concrete adaptive algorithm for the Biot's system. Such algorithms can be found in e.g., \cite{ErnSeb2009,RPEGK2017,ARN2019}. For
  the heat equation, readers are referred to e.g.,
  \cite{ChenFeng2004,Kreuzer2012,Verfurth2013,Kreuzer2019} for
  space-time adaptive algorithms as well as convergence analysis of
  adaptive methods. We point out that, once a space-time error
  indicator is available, a corresponding adaptive strategy follows and is largely independent
  of the equations.}

We now also present a new error estimator for mixed methods for time
dependent Darcy flow described by~\eqref{heat}. The fully discrete
scheme \eqref{fully} with $\bm{u}_h=\bm{v}=\bm{0}$ reduces to
\begin{subequations}\label{fullyheat}
\begin{align}
    c(\delta_tp_h^n,q)+d({\bm{w}}^n_h,q)&=({g}^n,q),\quad q\in Q^n_h,\\
    e({\bm{w}}^n_h,z)-d(\bm{z},{p}^n_h)&=0,\quad \bm{z}\in\bm{W}^n_h,
\end{align}
\end{subequations}
which obviously is a discretization of the heat equation or
time-dependent Darcy flow \eqref{heat}. Therefore, the a posteriori
analysis for \eqref{fully} directly applies to \eqref{fullyheat}. Here
$\bm{W}_h^n\times Q_h^n$ is the Raviart--Thomas and
Brezzi--Douglas--Marini mixed element space. Given an interval
$I\subseteq[0,T],$ we define the norm
$$\|(q,\bm{z})\|^2_{L^2(I;Y)}:=\int_I\big(\|q\|_c^2+\|\partial_tq\|_c^2+\|\bm{z}\|_{\bm{W}}^2+{\|\partial_t\bm{z}\|_{\bm{W}^\prime}^2}\big)ds.$$
Going through the proof of Theorems \ref{fullyresult} and \ref{lower}, $\bm{R}_1$ disappears when deriving the upper and lower bounds for the error of \eqref{fullyheat}. Therefore we obtain the following a posteriori error estimates.
\begin{corollary}
For $n=1,2,\ldots, N,$ the error of \eqref{fullyheat} satisfies
\begin{equation*}
    \begin{aligned}
    &\|E_p(t_n)\|_c^2+\|E_w(t_n)\|_e^2+\|(E_p,E_w)\|_{L^2(0,t_n;Y)}^2\\
    &\quad\lesssim\eta_{\text{init}}+\big(\sum_{i=1}^n\tau_i\widetilde{\mathcal{E}}^i_{\text{time}}+\tau_i(\tilde{\eta}_\text{space}^i)^\frac{1}{2}+\widetilde{\mathcal{E}}^i_{\text{data}}\big)^2\\
    &\quad+\sum_{i=1}^n\big(\tau_i\eta^i_{\text{time}}+\tau_i\eta^i_{\text{space}}+\int_{t_{i-1}}^{t_i}\|g-g_h^i\|^2dt\big),
    \end{aligned}
\end{equation*}
where
\begin{align*}
\eta_{\text{init}}&=\|p(0)-p_h^0\|_c^2+\|\bm{w}(0)-\bm{w}_h^0\|_e^2,\\
    \eta^i_{\text{time}}&=\|p_h^i-p_h^{i-1}\|^2_c+\|\bm{w}_h^i-\bm{w}_h^{i-1}\|^2_{\bm{W}},\\
    \tilde{\eta}_{\text{space}}^i&=\|g_h^i-\beta \delta_tp^i_{h}-\divg\bm{w}^i_h\|^2,\\
    \eta_{\text{space}}^i&=\tilde{\eta}_{\text{space}}^i+\mathcal{E}^3_{\Th^i}({p}^i_h,{\bm{w}}^i_h)+\mathcal{E}^3_{\Th^i\vee\Th^{i-1}}(\delta_tp^i_h,\delta_t{\bm{w}}^i_h).
\end{align*}
In addition, when $\bm{K}$ is a piecewise constant on $\Th^n$ it holds that,
\begin{equation*}
    \tau_n\eta^n_{\text{time}}+\tau_n\eta^n_{\text{space}}\lesssim\|(E_p,E_w)\|^2_{L^2(t_{n-1},t_n;Y)}+\int_{t_{i-1}}^{t_i}\|g-g_h^i\|^2dt.
\end{equation*}
\end{corollary}

{\section{Numerical examples}\label{secexp}
To support the theoretical results and show the behavior of the fully discrete error indicators,
we present a two-dimensional numerical example. 
The domain $\Omega=(0,1)\times (0,1)$ is the unit square in $\mathbb{R}^2$, and we consider the three field Biot's problem~\eqref{Biot} given
in~\S~\ref{secenergy}.
The rest of the setup is as follows.
\begin{itemize}
\item We set the Lam\'e parameters as $\lambda=\mu=0.4$.

\item $\beta=1$, $\alpha=1$, and $K=I\in \mathbb{R}^{2\times 2}$, the $2\times2$ identity matrix.
\item The analytic solution to the problem is
\begin{eqnarray*}
  &&  \bm{u}(t,\bm{x})=\cos t
     \begin{pmatrix}
       \sin(\pi x_1)\sin(\pi x_2)\\
       \sin(\pi x_1)\sin(\pi x_2)\\
     \end{pmatrix},\quad
     \bm{w}=-\nabla p= \pi\sin t
     \begin{pmatrix}
      \sin(\pi x_1)\cos(\pi x_2)\\
       \cos(\pi x_1)\sin(\pi x_2)\\
     \end{pmatrix},\\
&&  p(t,\bm{x})=\sin t\cos(\pi x_1)\cos(\pi x_2)
\end{eqnarray*}
\item As boundary conditions we take homogeneous Dirichlet condition for $\bm{u}$ and homogeneous Neumann condition for $p$. Correspondingly, $\Gamma_2=\emptyset$ in Section \ref{secenergy}.
\item We use a sequence of uniform triangular grids $\mathcal{T}_{h_k}$ with mesh sizes $h_k=2^{-k}$, $k=1:J$, $J=6$ or $J=7$. In Figure~\ref{fig:meshes} we have shown the coarsest mesh and a finer mesh.
  \begin{figure}[!h]
    \includegraphics[width=0.45\textwidth]{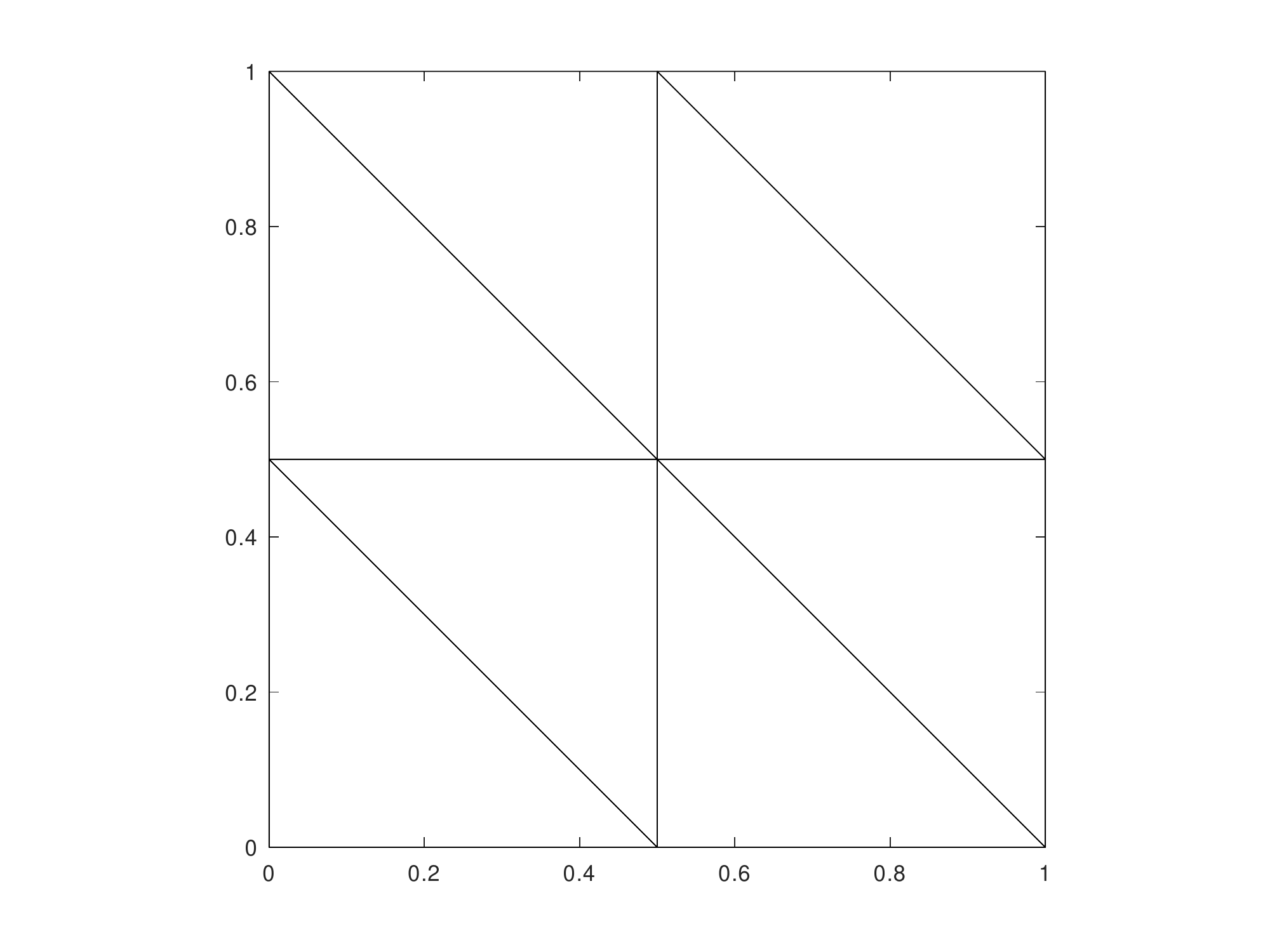}\hspace{0.1in}
    \includegraphics[width=0.45\textwidth]{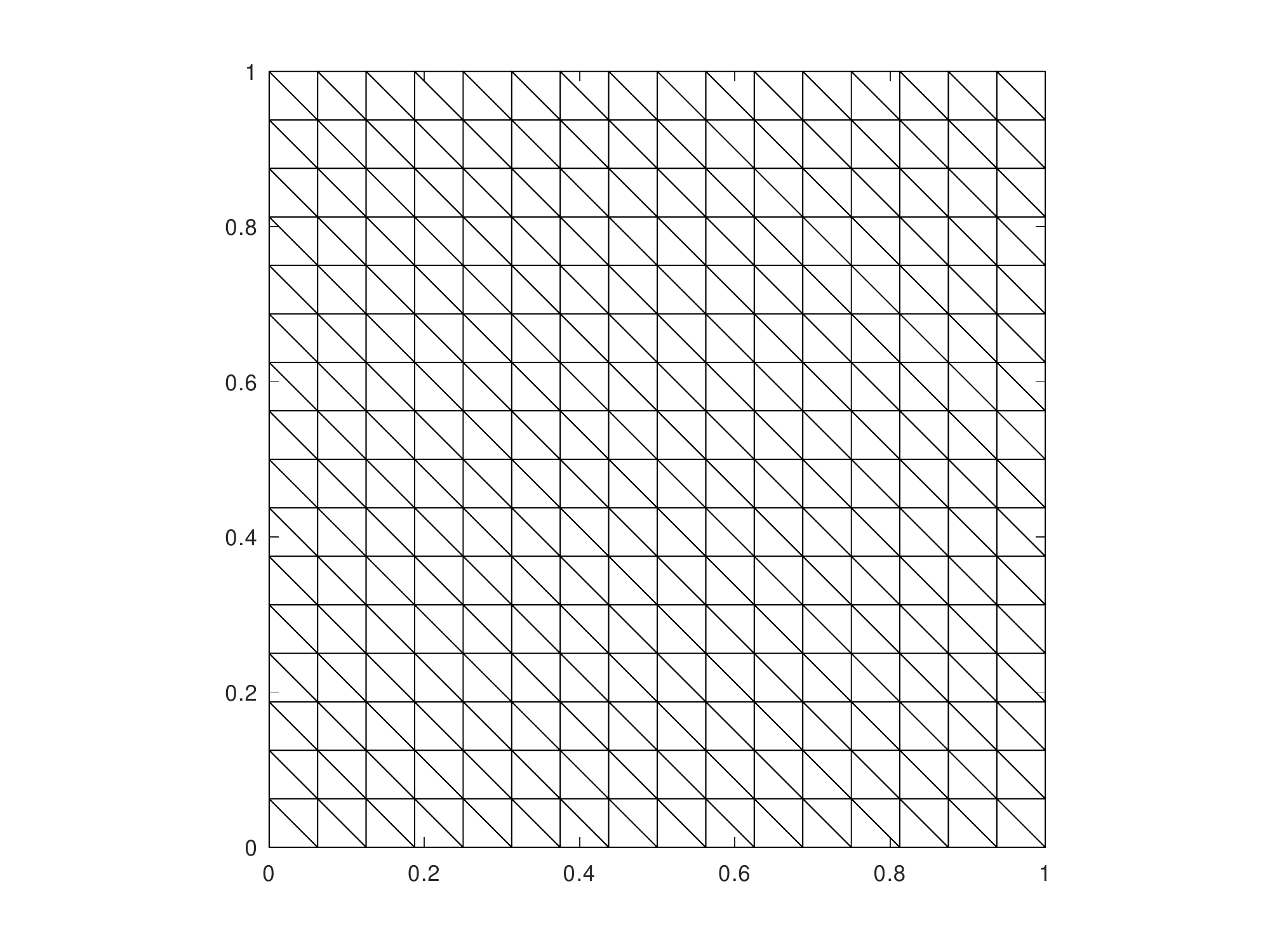}
    \caption{\label{fig:meshes}Coarsest mesh used in the experiments is
      with mesh size $h=2^{-1}$~(left). On the right we have shown a finer mesh with meshsize $h=2^{-4}$.}
    \end{figure}
\item The spatial discretization for \eqref{Biot} is based on the lowest order finite space triple $\bm{V}^0_{h_k}\times Q^0_{h_k}\times\bm{W}^0_{h_k}$ on grid $\mathcal{T}_{h_k}$ in Section \ref{secsemi}.
{\item We start the simulations at $t=0$ and reach the final
  time $T=1$ after $N=T/\tau$ time steps with uniform time step size $\tau$. We present two sets of tests: (1) a test where $h=h_k=2^{-k}$, $k=1:7$ and $\tau=\tau_k=0.4h_k$ change simultaneously; and (2) a test with fixed $\tau=5\times 10^{-5}$ and varying $h=h_k=2^{-k}$, $k=1:6$.
\item We set $p_h^0=0$, $\bm{w}_h^0=\bm{0}$, and choose $\bm{u}_h^{0}$ such that~\eqref{fully1} is satisfied
  at $t=0$, i.e. the equation \eqref{fully1} holds for $n=0$.
}
\item In computing the a priori error, we have omitted the term
  containing the  error term $\|\partial_t(\bm{w}-\bm{w}_h^\tau)\|^2_{\bm{W}^{\prime}}$ from the
  space-time norm $\|(\cdot,\cdot,\cdot)\|_{L^2(I;X)}$ and the modified norm is still denoted as $\|(\cdot,\cdot,\cdot)\|_{L^2(I;X)}$ by abuse of notation. This is
  motivated by the fact that such term only enters the computation of
  the a priori error and its computation is rather involved. However, the ratio between the true error and our indicator might be exaggerated by dropping that term.
\end{itemize}

We first present a simple test illustrating the
  efficiency of the indicator numerically by decreasing $h$ and the
  time step $\tau$ simultaneously.
Let $(\bm{u}_{h}^{\tau},p_{h_k}^{\tau},\bm{w}_{h_k}^{\tau})$ be the continuous temporal linear interpolant of the solution from \eqref{fully} with  $\mathcal{T}_h=\mathcal{T}_{h_k}$ and time step size $\tau=\tau_k=0.4h_k$.  In the first test, we compute the a priori error $E_k=\|(\bm{u}-\bm{u}_{h_k}^{\tau},p-p_{h_k}^{\tau},\bm{w}-\bm{w}_{h_k}^{\tau})\|_{L^2(0,T;X)}$
in Theorem \ref{fullyresult} with 5-point Gaussian quadrature on each time interval $[t_{n-1},t_n]$ and 25-point Gaussian quadrature on every element in $\mathcal{T}_{h_k}$. The global error $E_k$ is compared with following a posteriori error indicator
$$\mathcal{E}_{k} =\left(\sum_{n=1}^N\tau_k\mathcal{E}^n_{\text{time}}+\sum_{n=1}^N\tau_k\mathcal{E}^n_{\text{space}}\right)^{1/2}$$
on the stationary mesh $\mathcal{T}_{h_k}.$ Here we do not include the data oscillation $\mathcal{E}^n_{\text{data}}$ in $\mathcal{E}_k$ because it is generally small and dominated by other terms. { The results are shown in
  Table~\ref{t0}. As is seen from these results the convergence of the
  method is of order $(h+\tau)$. Moreover, the ratio between the a
  priori error and the value of the a posteriori error indicator approaches $2.5$.
 \begin{table}[!h]
   \centering
   \begin{tabular}{|l|l|l|l|l|l|}\hline\hline
     $h=2^{-k}$ & $E_k$ & $\mathcal{E}_k$ & $\mathcal{E}_k$/$E_k$
     & $E_{k-1}/E_k$ &$\mathcal{E}_{k-1}/\mathcal{E}_{k}$      \\ \hline
     $k=1$ &$0.947$ & $2.375$&  $2.51$  & N/A      &      N/A \\ \hline
     $k=2$ &$0.499$ & $1.250$&  $2.51$   & $1.906$ & $1.900$    \\ \hline
     $k=3$ &$0.253$ & $0.633$&  $2.50$   & $1.974$ &  $1.976$    \\ \hline
     $k=4$ &$0.127$ & $0.317$&  $2.50$     & $1.992$ &   $1.994$   \\ \hline
     $k=5$ &$0.064$ & $0.159$&  $2.50$    & $1.998$ & $1.998$   \\ \hline
     $k=6$ &$0.032$ & $0.079$&  $2.50$    & $2.000$ &    $2.000$   \\ \hline
     $k=7$ &$0.016$ & $0.039$&  $2.50$    & $2.000$ &    $2.000$   \\ \hline
   \end{tabular}
   \caption{A priori and a posteriori errors for simultaneously decreasing $h$ and $\tau$ with $\tau=0.4h$.\label{t0}}
 \end{table}
}

The next set of tests is for a fixed relatively small time step $\tau$
and aimed at comparing various characteristics of the indicators for different mesh sizes.
For a fixed time $t_n$, $n=1:N$ on a grid of size $h_k=2^{-k}$
we denote the ``true'' error at $t=t_n$ as
\begin{align*}
    e_k^n&=\big(\|\bm{u}(t_n)-\bm{u}_h^n\|_a^2+\|\partial_t\bm{u}(t_n)-\delta_t\bm{u}_h^n\|_a^2+\|p(t_n)-p_h^n\|_c^2\\
         &\quad+\|\partial_tp(t_n)-\delta_tp_h^n\|_c^2+\|\bm{w}(t_n)-\bm{w}_h^n\|_{\bm{W}}^2
           \big)^\frac{1}{2},
\end{align*}
which is again computed using 25-point Gaussian quadrature
element-wise. We compare $e_k^n$ with the following instantaneous
error estimator at $t_n$
\[
\varepsilon_k^n = \left(\mathcal{E}^n_{\text{time}}+\mathcal{E}^n_{\text{space}}\right)^{1/2}.
\]
Note that $E_k\approx\big(\sum_{n=1}^N\tau[e_k^n]^2\big)^\frac{1}{2}$ and $\mathcal{E}_k=\big(\sum_{n=1}^N\tau[\varepsilon_k^n]^2\big)^\frac{1}{2}$.

In Figure~\ref{f1} we show the plot of $\varepsilon_{k}^n$ for $n=1:N$ with $N=20,000$
for different mesh size $h_k$ as well as the ratios between the
indicators on two consecutive meshes, namely,
$\left(\varepsilon_k^n/\varepsilon_{k-1}^n\right)$. Similar behavior is observed in
Figure~\ref{f2}, where we have plotted the error reduction as predicted by the fully discrete error indicators.
\begin{figure}[!htb]
\includegraphics[width=0.45\textwidth]{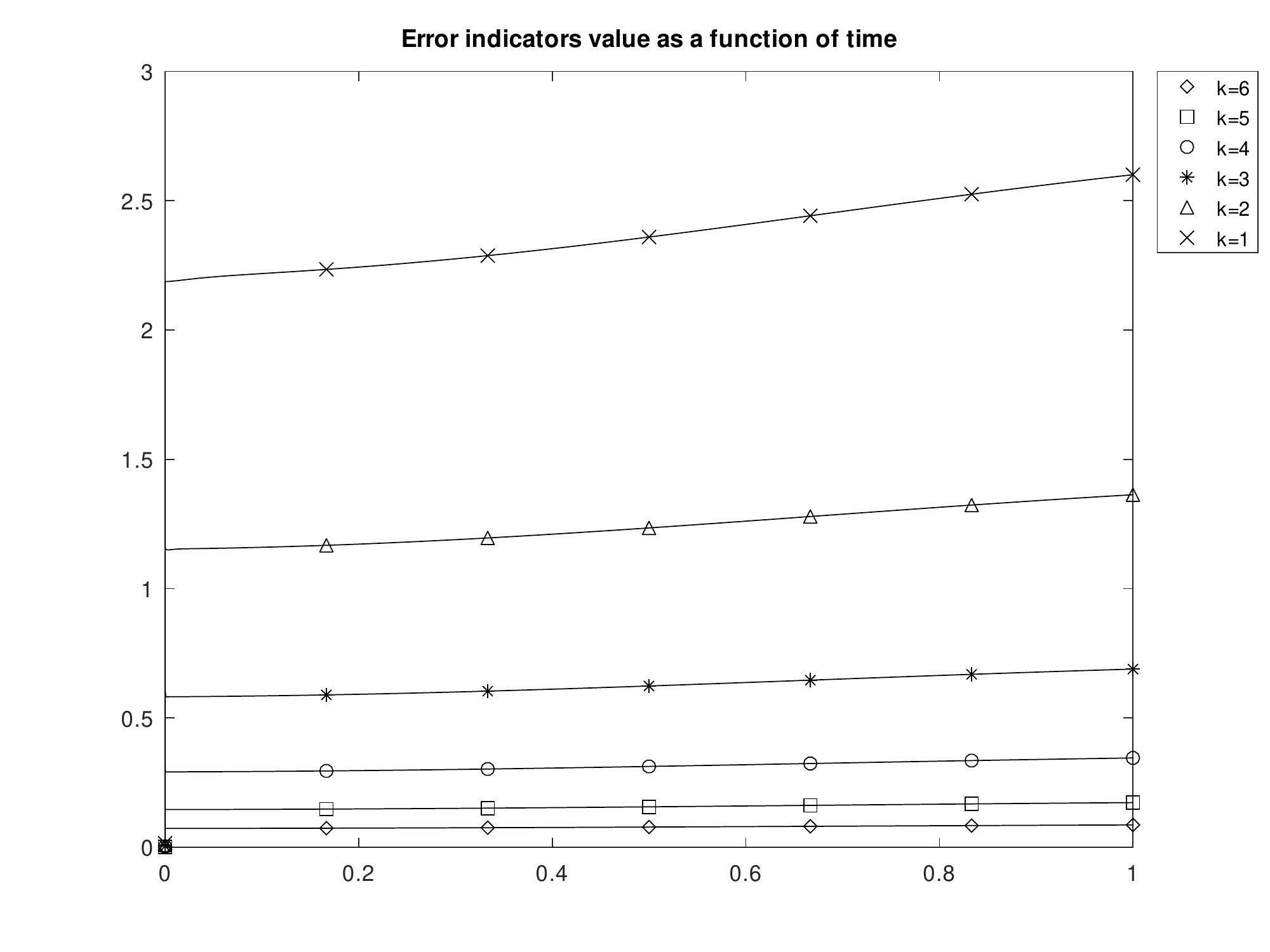}\hspace{0.1in}
\includegraphics[width=0.45\textwidth]{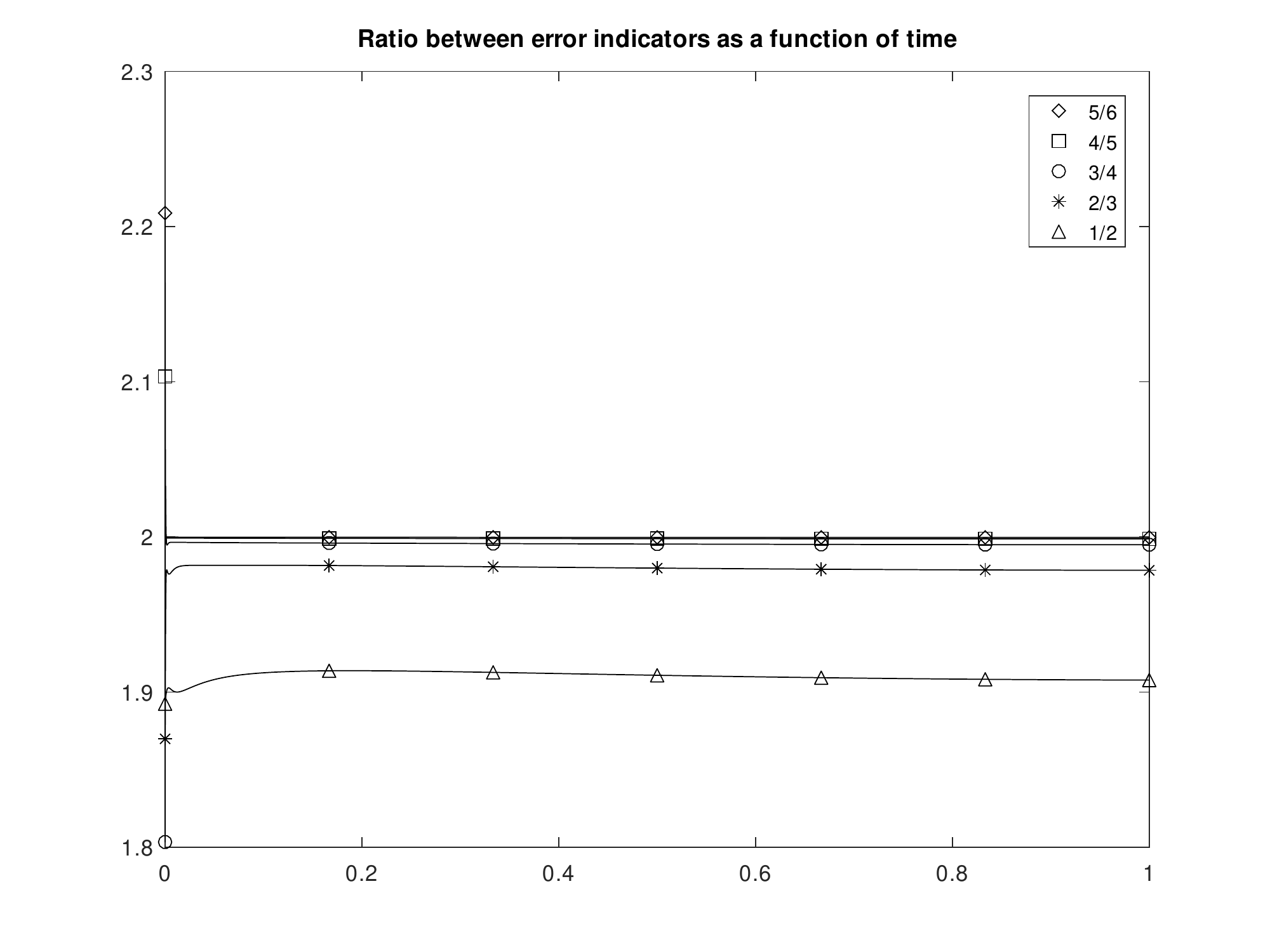}
\caption{\label{f1}Plot of $\{\varepsilon_k^n\}_{n=1}^N$, $k=1:6$ (left) and
    $\left(\varepsilon_{k-1}^n/\varepsilon_{k}^n\right)$ for $k=2:6$ (right)}
\end{figure}

\begin{figure}[!htb]
  \centering
\includegraphics[width=0.45\textwidth]{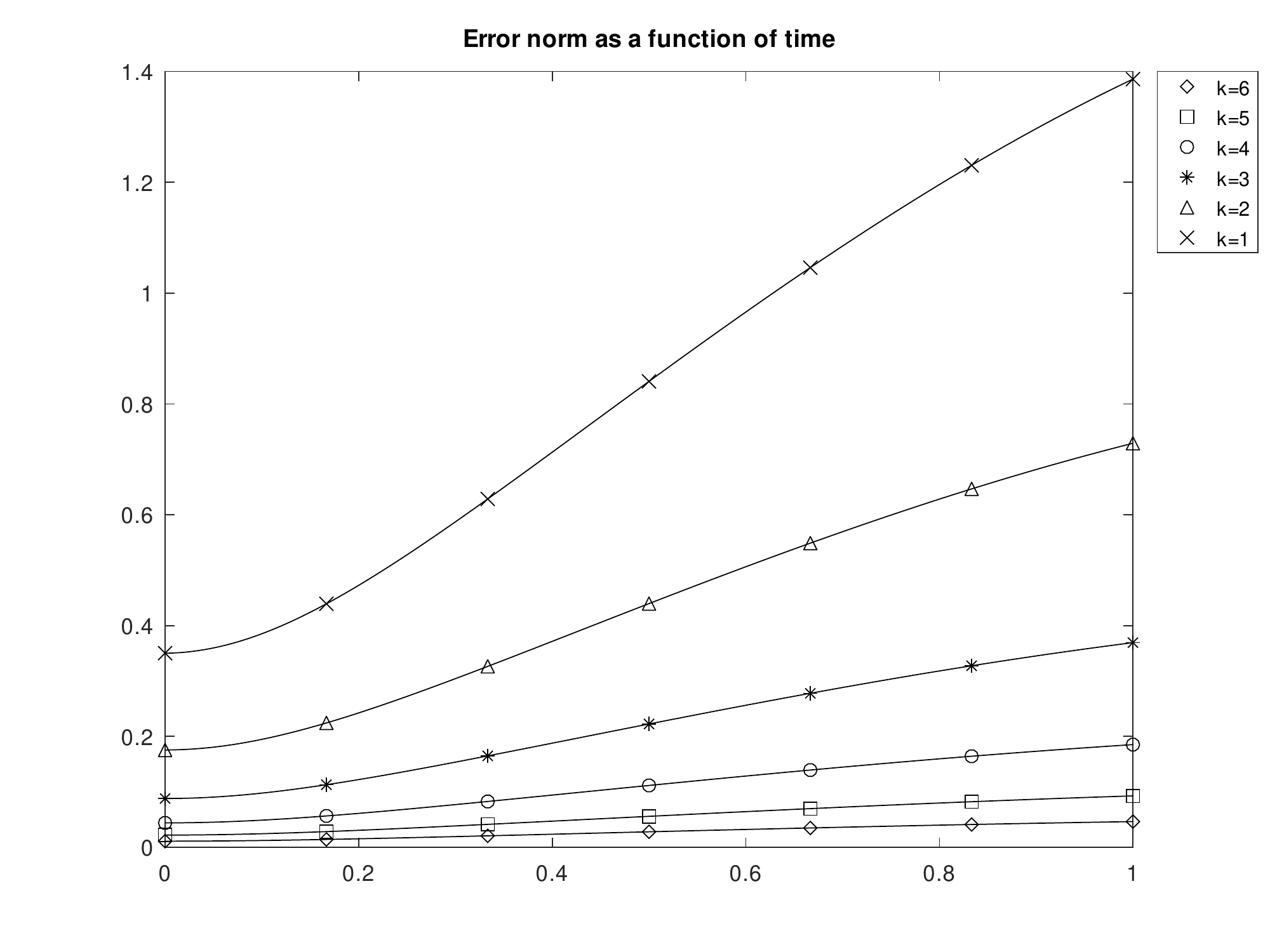}\hspace{0.1in}
\includegraphics[width=0.45\textwidth]{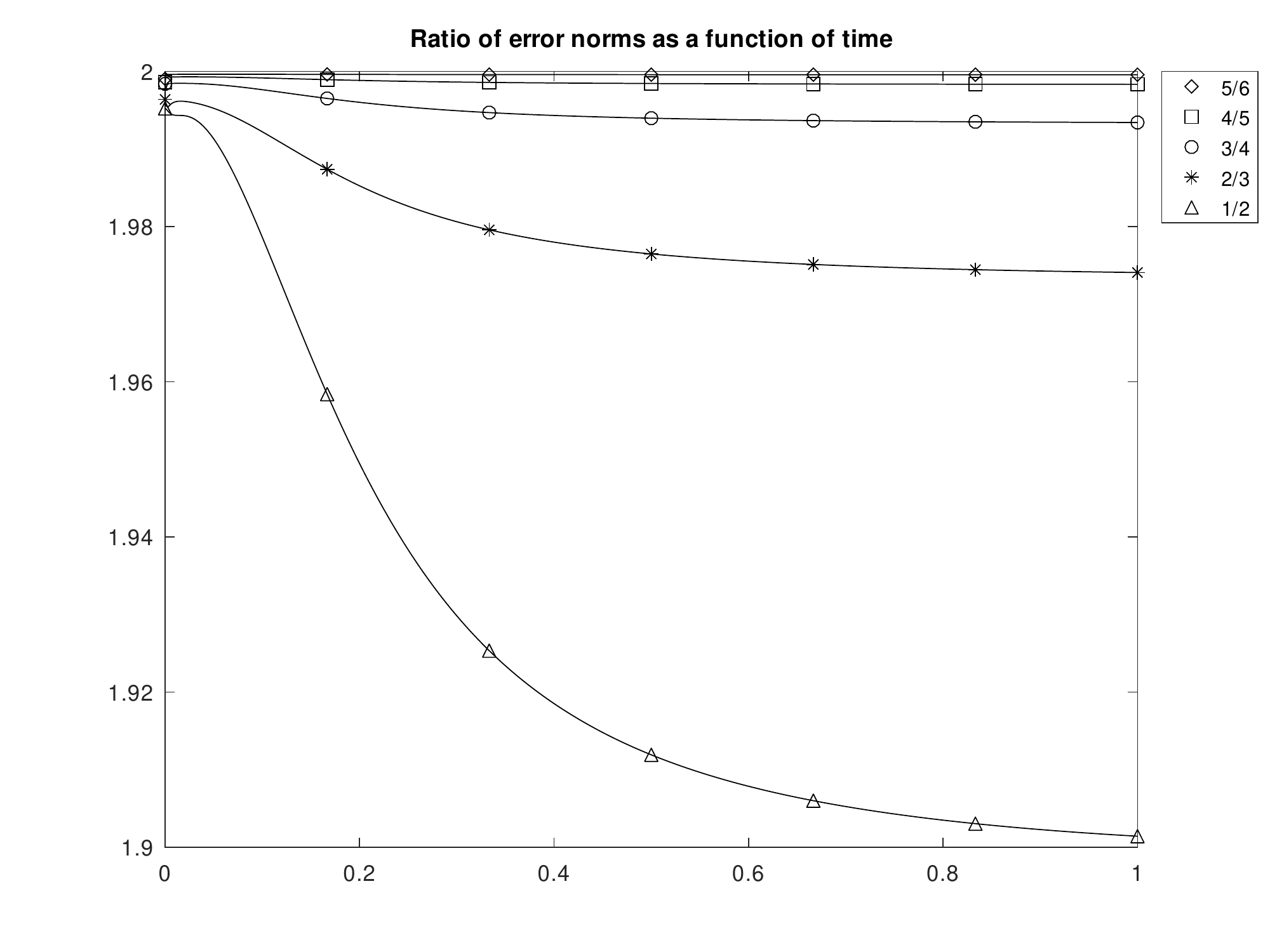}
\caption{\label{f2}Plot of $\{e_k^n\}_{n=1}^N$, $k=1:6$ (left) and
    $\left(e_{k-1}^n/e_{k}^n\right)$ for $k=2:6$ (right)}
\end{figure}

In Figure~\ref{f3} we plotted the ratio between the error indicators and the norm of the error as a function of time for varying mesh sizes.
As is seen from this plot, this ratio remains bounded. Notice that the theory developed earlier
shows reliability and efficiency when we integrate the indicators and the error norm over the interval $[0,T]$.
Such results (after integrations) are found in Table~\ref{t1} where
we illustrate the conclusions of Theorem~\ref{fullyresult}.
\begin{figure}[!htb]
  \centering
\includegraphics[width=0.5\textwidth]{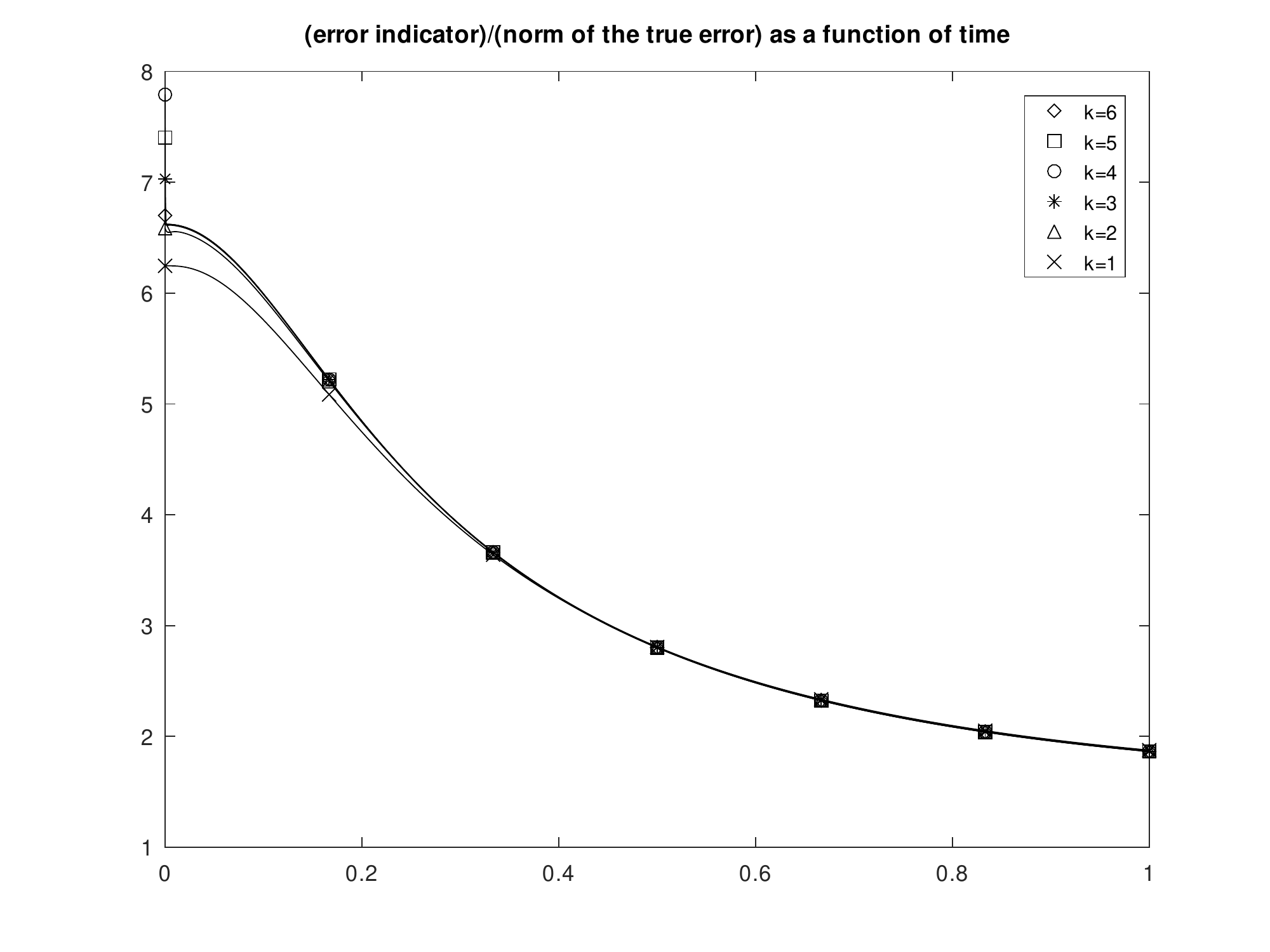}
\caption{\label{f3}Ratio between the values of the error indicators and the norm of the error as functions of time.}
\end{figure}
As seen from this table,
and expected from the theoretical results presented
earlier, the ratio between the indicators shows reduction by $2$. Moreover we see that the proposed fully discrete error indicators provide  two sided bounds for the error up to reasonable multiplicative constant.

\begin{table}\label{t1}
  \centering
  \begin{tabular}{|l|l|l|l|l|l|l|}\hline\hline
    $(h=2^{-k})$ &$k=1$ & $k=2$ &$k=3$ & $k=4$ &$k=5$ &$k=6$\\ \hline\hline 
$\mathcal{E}_{k-1}/\mathcal{E}_{k}$
 & N/A    & $1.910$ &  $1.976$ &  $1.994$ & $1.997$ & $2.003$ \\ \hline
 $E_{k-1}/E_{k}$
 & N/A    & $1.910$ &  $1.980$ &  $1.995$ & $1.999$ & $2.000$ \\ \hline
 $\mathcal{E}_{k}/E_{k}$
 & $2.63$ &  $2.63$ &   $2.63$ &  $2.62$  & $2.62$ & $2.62$ \\ \hline
  \end{tabular}
  \caption{Robustness and efficiency of the error indicators for $\tau=5\times 10^{-5}$ and varying mesh sizes.}
\end{table}
}

\section{Conclusion}\label{secconc}
In this paper, we obtain a two-sided residual a posteriori error estimator for the three-field mixed method in Biot's consolidation model. It is expected that our a posteriori error analysis generalizes to mixed methods for the five-field formulation based on weakly symmetric stress tensor (see \cite{LJJ2016}), although to present such a generalization might require an elaborate analysis. Combining our analysis with a posteriori error estimation of mixed methods for elasticity using strong symmetric stress (see e.g.,~\cite{CGG2019,CHHM2018,YL2019c}), we hope to obtain a two-sided residual estimator for the four-field formulation.

\section*{Acknowledgements}
The work of Zikatanov was supported in part by NSF grants DMS-1720114 and DMS-1819157.


\begin{thebibliography}{}

\bibitem[Ahmed {\em et~al.}(2019)Ahmed, Radu, \& Nordbotten]{ARN2019}
{\sc Ahmed, E., Radu, F.~A. \& Nordbotten, J.~M.} (2019)
\newblock Adaptive poromechanics computations based on a posteriori error
  estimates for fully mixed formulations of {B}iot's consolidation model.
\newblock {\em Comput. Methods Appl. Mech. Engrg.}, {\bf 347}, 264--294.

\bibitem[Ahmed {\em et~al.}(2020)Ahmed, Nordbotten, \& Radu]{ARN2020}
{\sc Ahmed, E., Nordbotten, J.~M. \& Radu, F.~A.} (2020)
\newblock Adaptive asynchronous time-stepping, stopping criteria, and a
  posteriori error estimates for fixed-stress iterative schemes for coupled
  poromechanics problems.
\newblock {\em J. Comput. Appl. Math.}, {\bf 364}, 112312, 25.

\bibitem[Alonso(1996)Alonso]{Alonso1996}
{\sc Alonso, A.} (1996)
\newblock Error estimators for a mixed method.
\newblock {\em Numer. Math.}, {\bf 74}, 385--395.

\bibitem[{B}iot(1941){B}iot]{biot1}
{\sc {B}iot, M.~A.} (1941)
\newblock General theory of three-dimensional consolidation.
\newblock {\em Journal of Applied Physics\/}, {\bf 12}, 155--164.

\bibitem[{B}iot(1955){B}iot]{biot2}
{\sc {B}iot, M.~A.} (1955)
\newblock Theory of elasticity and consolidation for a porous anisotropic
  solid.
\newblock {\em Journal of Applied Physics\/}, {\bf 26}, 182--185.

\bibitem[Brezzi {\em et~al.}(1985)Brezzi, {Douglas Jr.}, \& Marini]{BDM1985}
{\sc Brezzi, F., {Douglas Jr.}, J. \& Marini, L.~D.} (1985)
\newblock Two families of mixed finite elements for second order elliptic
  problems.
\newblock {\em Numer. Math.}, {\bf 2}, 217--235.

\bibitem[Carstensen {\em et~al.}(2019)Carstensen, Gallistl, \&
  Gedicke]{CGG2019}
{\sc Carstensen, C., Gallistl, D. \& Gedicke, J.} (2019)
\newblock Residual-based a posteriori error analysis for symmetric mixed
  {A}rnold-{W}inther {FEM}.
\newblock {\em Numer. Math.}, {\bf 142}, 205--234.

\bibitem[Casc\'{o}n {\em et~al.}(2006)Casc\'{o}n, Ferragut, \&
  Asensio]{CFA2006}
{\sc Casc\'{o}n, J.~M., Ferragut, L. \& Asensio, M.~I.} (2006)
\newblock Space-time adaptive algorithm for the mixed parabolic problem.
\newblock {\em Numer. Math.}, {\bf 103}, 367--392.

\bibitem[Cascon {\em et~al.}(2007)Cascon, Nochetto, \& Siebert]{CNS2007}
{\sc Cascon, J.~M., Nochetto, R.~H. \& Siebert, K.~G.} (2007)
\newblock Design and convergence of {AFEM} in {$H({\rm div})$}.
\newblock {\em Math. Models Methods Appl. Sci.}, {\bf 17}, 1849--1881.

\bibitem[Cascon {\em et~al.}(2008)Cascon, Kreuzer, Nochetto, \&
  Siebert]{CKNS2008}
{\sc Cascon, J.~M., Kreuzer, C., Nochetto, R.~H. \& Siebert, K.~G.} (2008)
\newblock Quasi-optimal convergence rate for an adaptive finite element method.
\newblock {\em SIAM J. Numer. Anal.}, {\bf 46}, 2524--2550.

\bibitem[Chen {\em et~al.}(2018)Chen, Hu, Huang, \& Man]{CHHM2018}
{\sc Chen, L., Hu, J., Huang, X. \& Man, H.} (2018)
\newblock Residual-based a posteriori error estimates for symmetric conforming
  mixed finite elements for linear elasticity problems.
\newblock {\em Sci. China Math.}, {\bf 61}, 973--992.

\bibitem[Chen \& Wu(2017)Chen \& Wu]{CW2017}
{\sc Chen, L. \& Wu, Y.} (2017)
\newblock Convergence of adaptive mixed finite element methods for the {H}odge
  {L}aplacian equation: without harmonic forms.
\newblock {\em SIAM J. Numer. Anal.}, {\bf 55}, 2905--2929.

\bibitem[Chen \& Feng(2004)Chen \& Feng]{ChenFeng2004}
{\sc Chen, Z. \& Feng, J.} (2004)
\newblock An adaptive finite element algorithm with reliable and efficient
  error control for linear parabolic problems.
\newblock {\em Math. Comp.}, {\bf 73}, 1167--1193.

\bibitem[Cl\'{e}ment(1975)Cl\'{e}ment]{Clement}
{\sc Cl\'{e}ment, P.} (1975)
\newblock Approximation by finite element functions using local regularization.
\newblock {\em Rev. Fran\c{c}aise Automat. Informat. Recherche
  Op\'{e}rationnelle S\'{e}r.}, {\bf 9}, 77--84.

\bibitem[Demlow \& Hirani(2014)Demlow \& Hirani]{DH2014}
{\sc Demlow, A. \& Hirani, A.~N.} (2014)
\newblock A posteriori error estimates for finite element exterior calculus:
  the de {R}ham complex.
\newblock {\em Found. Comput. Math.}, {\bf 14}, 1337--1371.

\bibitem[Diening {\em et~al.}(2016)Diening, Kreuzer, \& Stevenson]{DKS2016}
{\sc Diening, L., Kreuzer, C. \& Stevenson, R.} (2016)
\newblock Instance optimality of the adaptive maximum strategy.
\newblock {\em Found. Comput. Math.}, {\bf 16}, 33--68.

\bibitem[Eriksson \& Johnson(1991)Eriksson \& Johnson]{Johnson1991}
{\sc Eriksson, K. \& Johnson, C.} (1991)
\newblock Adaptive finite element methods for parabolic problems. {I}. {A}
  linear model problem.
\newblock {\em SIAM J. Numer. Anal.}, {\bf 28}, 43--77.

\bibitem[Eriksson \& Johnson(1995)Eriksson \& Johnson]{Johnson1995}
{\sc Eriksson, K. \& Johnson, C.} (1995)
\newblock Adaptive finite element methods for parabolic problems. {II}.
  {O}ptimal error estimates in {$L_\infty L_2$} and {$L_\infty L_\infty$}.
\newblock {\em SIAM J. Numer. Anal.}, {\bf 32}, 706--740.

\bibitem[Ern {\em et~al.}(2019)Ern, Smears, \& Vohral\'{\i}k]{Ern2019}
{\sc Ern, A., Smears, I. \& Vohral\'{\i}k, M.} (2019)
\newblock Equilibrated flux {\it a posteriori} error estimates in
  {$L^2(H^1)$}-norms for high-order discretizations of parabolic problems.
\newblock {\em IMA J. Numer. Anal.}, {\bf 39}, 1158--1179.

\bibitem[Ern \& Meunier(2009)Ern \& Meunier]{ErnSeb2009}
{\sc Ern, A. \& Meunier, S.} (2009)
\newblock A posteriori error analysis of {E}uler-{G}alerkin approximations to
  coupled elliptic-parabolic problems.
\newblock {\em M2AN Math. Model. Numer. Anal.}, {\bf 42}, 353--375.

\bibitem[Ern \& Vohral\'{\i}k(2010)Ern \& Vohral\'{\i}k]{Ern2010}
{\sc Ern, A. \& Vohral\'{\i}k, M.} (2010)
\newblock A posteriori error estimation based on potential and flux
  reconstruction for the heat equation.
\newblock {\em SIAM J. Numer. Anal.}, {\bf 48}, 198--223.

\bibitem[Ern \& Vohral\'ik(2015)Ern \& Vohral\'ik]{Ern2015}
{\sc Ern, A. \& Vohral\'ik, M.} (2015)
\newblock Polynomial-degree-robust a posteriori estimates in a unified setting
  for conforming, nonconforming, discontinuous galerkin, and mixed
  discretizations.
\newblock {\em SIAM J. Numer. Anal.}, {\bf 53}, 1058--1081.

\bibitem[Evans(2010)Evans]{Evans2010}
{\sc Evans, L.~C.} (2010)
\newblock {\em Partial differential equations\/}. Graduate Studies in
  Mathematics,  vol.~19, second edn.
\newblock Providence, RI: American Mathematical Society, pp. xxii+749.

\bibitem[Gaspar {\em et~al.}(2003)Gaspar, Lisbona, \&
  Vabishchevich]{Gaspar2003}
{\sc Gaspar, F.~J., Lisbona, F.~J. \& Vabishchevich, P.~N.} (2003)
\newblock A finite difference analysis of {B}iot's consolidation model.
\newblock {\em Appl. Numer. Math.}, {\bf 44}, 487--506.

\bibitem[Gaspar {\em et~al.}(2006)Gaspar, Lisbona, \&
  Vabishchevich]{Gaspar2006}
{\sc Gaspar, F.~J., Lisbona, F.~J. \& Vabishchevich, P.~N.} (2006)
\newblock Staggered grid discretizations for the quasi-static {B}iot's
  consolidation problem.
\newblock {\em Appl. Numer. Math.}, {\bf 56}, 888--898.

\bibitem[Gaspoz {\em et~al.}(2019)Gaspoz, Siebert, Kreuzer, \&
  Ziegler]{Kreuzer2019}
{\sc Gaspoz, F.~D., Siebert, K., Kreuzer, C. \& Ziegler, D.~A.} (2019)
\newblock A convergent time-space adaptive {${\rm dG}(s)$} finite element
  method for parabolic problems motivated by equal error distribution.
\newblock {\em IMA J. Numer. Anal.}, {\bf 39}, 650--686.

\bibitem[Girault \& Raviart(1986)Girault \& Raviart]{GR1986}
{\sc Girault, V. \& Raviart, P.-A.} (1986)
\newblock {\em Finite element methods for {N}avier-{S}tokes equations\/}.
  Springer Series in Computational Mathematics,  vol.~5.
\newblock Berlin: Springer-Verlag, pp. x+374.
\newblock Theory and algorithms.

\bibitem[Hiptmair(2002)Hiptmair]{Hip2002}
{\sc Hiptmair, R.} (2002)
\newblock Finite elements in computational electromagnetism.
\newblock {\em Acta Numer.}, {\bf 11}, 237--339.

\bibitem[Holst {\em et~al.}(2020)Holst, Li, Mihalik, \& Szypowski]{HLMS2019}
{\sc Holst, M., Li, Y., Mihalik, A. \& Szypowski, R.} (2020)
\newblock Convergence and optimality of adaptive mixed methods for {P}oisson's
  equation in the {FEEC} framework.
\newblock {\em J. Comp. Math.}, {\bf 38}, 748--767.

\bibitem[Hong \& Kraus(2018)Hong \& Kraus]{HongKraus2018}
{\sc Hong, Q. \& Kraus, J.} (2018)
\newblock Parameter-robust stability of classical three-field formulation of
  {B}iot's consolidation model.
\newblock {\em Electron. Trans. Numer. Anal.}, {\bf 48}, 202--226.

\bibitem[Hu {\em et~al.}(2017)Hu, Rodrigo, Gaspar, \& Zikatanov]{HRGZ2017}
{\sc Hu, X., Rodrigo, C., Gaspar, F.~J. \& Zikatanov, L.~T.} (2017)
\newblock A nonconforming finite element method for the {B}iot's consolidation
  model in poroelasticity.
\newblock {\em Journal of Computational and Applied Mathematics\/}, {\bf 310},
  143 -- 154.

\bibitem[Huang \& Xu(2012)Huang \& Xu]{HuangXu2012}
{\sc Huang, J. \& Xu, Y.} (2012)
\newblock Convergence and complexity of arbitrary order adaptive mixed element
  methods for the poisson equation.
\newblock {\em Sci. China Math.}, {\bf 55}, 1083--1098.

\bibitem[Kim {\em et~al.}(2018)Kim, Park, \& Seo]{KPS2018}
{\sc Kim, D., Park, E.-J. \& Seo, B.} (2018)
\newblock Space-time adaptive methods for the mixed formulation of a linear
  parabolic {P}roblem.
\newblock {\em J. Sci. Comput.}, {\bf 74}, 1725--1756.

\bibitem[Kondratiev \& Oleinik(1989)Kondratiev \& Oleinik]{KO1989}
{\sc Kondratiev, V.~A. \& Oleinik, O.~A.} (1989)
\newblock On {K}orn's inequalities.
\newblock {\em C. R. Acad. Sci. Paris S\'{e}r. I Math.}, {\bf 308}, 483--487.

\bibitem[Kreuzer {\em et~al.}(2012)Kreuzer, M\"{o}ller, Schmidt, \&
  Siebert]{Kreuzer2012}
{\sc Kreuzer, C., M\"{o}ller, C.~A., Schmidt, A. \& Siebert, K.~G.} (2012)
\newblock Design and convergence analysis for an adaptive discretization of the
  heat equation.
\newblock {\em IMA J. Numer. Anal.}, {\bf 32}, 1375--1403.

\bibitem[Kumar {\em et~al.}(2018)Kumar, Matculevich, Nordbotten, \&
  Repin]{Kumar2018}
{\sc Kumar, K., Matculevich, S., Nordbotten, J. \& Repin, S.} (2018)
\newblock {Guaranteed and computable bounds of approximation errors for the
  semi-discrete Biot problem}.
\newblock {\em arXiv e-prints\/}, arXiv:1808.08036.

\bibitem[Lakkis \& Makridakis(2006)Lakkis \& Makridakis]{Lakkis2006}
{\sc Lakkis, O. \& Makridakis, C.} (2006)
\newblock Elliptic reconstruction and a posteriori error estimates for fully
  discrete linear parabolic problems.
\newblock {\em Math. Comp.}, {\bf 75}, 1627--1658.

\bibitem[Larson \& M\aa{}lqvist(2011)Larson \& M\aa{}lqvist]{Larson2011}
{\sc Larson, M.~G. \& M\aa{}lqvist, A.} (2011)
\newblock A posteriori error estimates for mixed finite element approximations
  of parabolic problems.
\newblock {\em Numer. Math.}, {\bf 118}, 33--48.

\bibitem[Lee(2016)Lee]{LJJ2016}
{\sc Lee, J.~J.} (2016)
\newblock Robust error analysis of coupled mixed methods for {B}iot's
  consolidation model.
\newblock {\em J. Sci. Comput.}, {\bf 69}, 610--632.

\bibitem[Lee {\em et~al.}(2017)Lee, Mardal, \& Winther]{LMW2017}
{\sc Lee, J.~J., Mardal, K.-A. \& Winther, R.} (2017)
\newblock Parameter-robust discretization and preconditioning of {B}iot's
  consolidation model.
\newblock {\em SIAM J. Sci. Comput.}, {\bf 39}, A1--A24.

\bibitem[{Li}(2019a){Li}]{YL2019c}
{\sc {Li}, Y.} (2019a)
\newblock {Quasi-optimal adaptive hybridized mixed finite element methods for
  linear elasticity}.
\newblock {\em arXiv e-prints\/}.

\bibitem[{Li}(2019b){Li}]{YL2019b}
{\sc {Li}, Y.} (2019b)
\newblock {Quasi-optimal adaptive mixed finite element methods for controlling
  natural norm errors}.
\newblock {\em arXiv e-prints\/}, arXiv:1907.03852, to appear in Math.~Comp.

\bibitem[Li(2019)Li]{YL2019}
{\sc Li, Y.} (2019)
\newblock Some convergence and optimality results of adaptive mixed methods in
  finite element exterior calculus.
\newblock {\em SIAM J. Numer. Anal.}, {\bf 57}, 2019--2042.

\bibitem[Makridakis \& Nochetto(2003)Makridakis \& Nochetto]{Makridakis2003}
{\sc Makridakis, C. \& Nochetto, R.~H.} (2003)
\newblock Elliptic reconstruction and a posteriori error estimates for
  parabolic problems.
\newblock {\em SIAM J. Numer. Anal.}, {\bf 41}, 1585--1594.

\bibitem[Memon {\em et~al.}(2012)Memon, Nataraj, \& Pani]{Memon2012}
{\sc Memon, S., Nataraj, N. \& Pani, A.~K.} (2012)
\newblock An a posteriori error analysis of mixed finite element {G}alerkin
  approximations to second order linear parabolic problems.
\newblock {\em SIAM J. Numer. Anal.}, {\bf 50}, 1367--1393.

\bibitem[Murad {\em et~al.}(1996)Murad, Thom\'{e}e, \& Loula]{Murad1996}
{\sc Murad, M.~A., Thom\'{e}e, V. \& Loula, A. F.~D.} (1996)
\newblock Asymptotic behavior of semidiscrete finite-element approximations of
  {B}iot's consolidation problem.
\newblock {\em SIAM J. Numer. Anal.}, {\bf 33}, 1065--1083.

\bibitem[N\'ed\'elec(1980)N\'ed\'elec]{Nedelec1980}
{\sc N\'ed\'elec, J.-C.} (1980)
\newblock Mixed finite elements in $\mathbf{R}^3$.
\newblock {\em Numer. Math.}, {\bf 35}, 315--341.

\bibitem[Nordbotten(2016)Nordbotten]{nordbotten_FVM}
{\sc Nordbotten, J.~M.} (2016)
\newblock Stable cell-centered finite volume discretization for {B}iot
  equations.
\newblock {\em SIAM Journal on Numerical Analysis\/}, {\bf 54}, 942--968.

\bibitem[Oyarz\'ua \& Ruiz-Baier(2016)Oyarz\'ua \& Ruiz-Baier]{OyRu2016}
{\sc Oyarz\'ua, R. \& Ruiz-Baier, R.} (2016)
\newblock Locking-free finite element methods for poroelasticity.
\newblock {\em SIAM J. Numer. Anal.}, {\bf 54}, 2951--2973.

\bibitem[Pasciak \& Zhao(2002)Pasciak \& Zhao]{PZ2002}
{\sc Pasciak, J.~E. \& Zhao, J.} (2002)
\newblock Overlapping schwarz methods in {H(curl)} on polyhedral domains.
\newblock {\em J. Numer. Math.}, {\bf 10}, 221--234.

\bibitem[Phillips \& Wheeler(2007a)Phillips \& Wheeler]{Wheeler2007I}
{\sc Phillips, P.~J. \& Wheeler, M.~F.} (2007a)
\newblock A coupling of mixed and continuous {G}alerkin finite element methods
  for poroelasticity. {I}. {T}he continuous in time case.
\newblock {\em Comput. Geosci.}, {\bf 11}, 131--144.

\bibitem[Phillips \& Wheeler(2007b)Phillips \& Wheeler]{Wheeler2007II}
{\sc Phillips, P.~J. \& Wheeler, M.~F.} (2007b)
\newblock A coupling of mixed and continuous {G}alerkin finite element methods
  for poroelasticity. {II}. {T}he discrete-in-time case.
\newblock {\em Comput. Geosci.}, {\bf 11}, 145--158.

\bibitem[Picasso(1998)Picasso]{Picasso1998}
{\sc Picasso, M.} (1998)
\newblock Adaptive finite elements for a linear parabolic problem.
\newblock {\em Comput. Methods Appl. Mech. Engrg.}, {\bf 167}, 223--237.

\bibitem[Raviart \& Thomas(1977)Raviart \& Thomas]{RT1977}
{\sc Raviart, P.-A. \& Thomas, J.~M.} (1977)
\newblock A mixed finite element method for 2nd order elliptic problems.
\newblock {\em Mathematical aspects of finite element methods ({P}roc. {C}onf.,
  {C}onsiglio {N}az. delle {R}icerche ({C}.{N}.{R}.), {R}ome, 1975)\/}. Lecture
  Notes in Math., Vol. 606.
\newblock Berlin: Springer, pp. 292--315.

\bibitem[Riedlbeck {\em et~al.}(2017)Riedlbeck, Di~Pietro, Ern, Granet, \&
  Kazymyrenko]{RPEGK2017}
{\sc Riedlbeck, R., Di~Pietro, D.~A., Ern, A., Granet, S. \& Kazymyrenko, K.}
  (2017)
\newblock Stress and flux reconstruction in {B}iot's poro-elasticity problem
  with application to a posteriori error analysis.
\newblock {\em Comput. Math. Appl.}, {\bf 73}, 1593--1610.

\bibitem[Rodrigo {\em et~al.}(2018)Rodrigo, Hu, Ohm, Adler, Gaspar, \&
  Zikatanov]{RHOAGZ2018}
{\sc Rodrigo, C., Hu, X., Ohm, P., Adler, J.~H., Gaspar, F.~J. \& Zikatanov,
  L.~T.} (2018)
\newblock New stabilized discretizations for poroelasticity and the {S}tokes'
  equations.
\newblock {\em Comput. Methods Appl. Mech. Engrg.}, {\bf 341}, 467--484.

\bibitem[Sch\"{o}berl(2008)Sch\"{o}berl]{Schoberl2008}
{\sc Sch\"{o}berl, J.} (2008)
\newblock A posteriori error estimates for {M}axwell equations.
\newblock {\em Math. Comp.}, {\bf 77}, 633--649.

\bibitem[Showalter(2000)Showalter]{Showalter2000}
{\sc Showalter, R.~E.} (2000)
\newblock Diffusion in poro-elastic media.
\newblock {\em J. Math. Anal. Appl.}, {\bf 251}, 310--340.

\bibitem[Terzaghi(1943)Terzaghi]{terzaghi}
{\sc Terzaghi, K.} (1943)
\newblock {\em Theoretical Soil Mechanics\/}.
\newblock New York: John Wiley\&Sons, Inc.

\bibitem[Verf\"{u}rth(1991)Verf\"{u}rth]{Verfurth1991}
{\sc Verf\"{u}rth, R.} (1991)
\newblock A posteriori error estimators for the {S}tokes equations. {II}.
  {N}onconforming discretizations.
\newblock {\em Numer. Math.}, {\bf 60}, 235--249.

\bibitem[Verf\"{u}rth(2003)Verf\"{u}rth]{Verfurth2003}
{\sc Verf\"{u}rth, R.} (2003)
\newblock A posteriori error estimates for finite element discretizations of
  the heat equation.
\newblock {\em Calcolo\/}, {\bf 40}, 195--212.

\bibitem[Verf\"{u}rth(2013)Verf\"{u}rth]{Verfurth2013}
{\sc Verf\"{u}rth, R.} (2013)
\newblock {\em A posteriori error estimation techniques for finite element
  methods\/}.
\newblock Numerical Mathematics and Scientific Computation.
\newblock Oxford: Oxford University Press, pp. xx+393.

\bibitem[\v{Z}en\'{\i}\v{s}ek(1984)\v{Z}en\'{\i}\v{s}ek]{Zen1984}
{\sc \v{Z}en\'{\i}\v{s}ek, A.} (1984)
\newblock The existence and uniqueness theorem in {B}iot's consolidation
  theory.
\newblock {\em Apl. Mat.}, {\bf 29}, 194--211.

\end{thebibliography}
\end{document}